\documentclass{elsarticle}

\usepackage[latin1]{inputenc}
\usepackage[english]{babel}
\usepackage{graphicx}
\usepackage{amsmath,amsfonts,amssymb,amsthm}
\usepackage{bbm,wasysym,stmaryrd}
\usepackage{subfigure}
\usepackage{color}
\usepackage[normalem]{ulem}

\journal{Physica D}

\newcommand{\R}{\mathbbm{R}}

\newcommand{\ind}[1]{\mathbbm{1}_{#1}}
\newcommand{\derpart}[2]{ \frac{\partial #1}{\partial #2} } 
\newcommand{\der}[2]{ \frac{\text{d} #1}{\text{d} #2} }  
\newcommand{\Exp}{\mathbbm{E}} 

\newcommand{\Lop}{\mathcal{L}}

\newtheorem{thm}{Theorem}
\newtheorem{theorem}{Theorem}
\newtheorem{proposition}{Proposition}
\newtheorem{corollary}{Corollary}
\newtheorem{definition}{Definition}
\newtheorem{lemma}{Lemma}
\newtheorem{example}{Example}

\newenvironment{remark}{\par {\noindent \it \sc Remark.} \small \it } {}

\definecolor{orange}{rgb}{.30,0.20,1.0}

\graphicspath{{../}}

\begin{document}
	\begin{frontmatter}
	
	\title{Dynamics and absorption properties of stochastic equations with H\"older diffusion coefficients} 

	\author[John]{Jonathan Touboul}
	  \author[Gilles]{Gilles Wainrib}
	\address[John]{ The Mathematical Neuroscience Laboratory, CIRB\footnote{CNRS UMR 7241,
	INSERM U1050, UPMC ED 158, MEMOLIFE PSL*}, Coll\`ege de France\@, 11, place Marcelin Berthelot, 75005 Paris, France and	INRIA Paris-Rocquencourt, Mycenae Team\\
	Phone: (+33) 1 44 27 13 88, jonathan.touboul@college-de-france.fr}
	\address[Gilles]{{LAGA, Universit\'e Paris 13, Villetaneuse,
	    France, and \\
	Ecole Normale Sup\'erieure, D\'epartement d'Informatique (DATA), 45 rue d'Ulm, 75005 Paris, France}}
	
	\begin{abstract}
	{In this article, we characterize the dynamics and absorption properties of a class of stochastic differential equations around singular points where both the drift and diffusion functions vanish. According to the H\"older coefficient $\alpha$ of the diffusion function around the singular point, we identify different regimes: a regime where the solutions almost surely reach the singular point in finite time, and regimes of exponential attraction or repulsion from the singular point. Stability of the absorbing state, large deviations for the absorption time, existence of stationary or quasi-stationary distributions are discussed. In particular, we show that quasi-stationary distributions only exist for $\alpha<3/4$, and for $\alpha \in (3/4,1)$, no quasi-stationary distribution is found and numerical simulations tend to show that the process conditioned on not being absorbed initiates an almost sure exponential convergence towards the absorbing state (as is demonstrated to be true for $\alpha=1$). These results have several implications in for the understanding of stochastic bifurcations, and we completely unfold two generic situations: the pitchfork and saddle-node bifurcations, and discuss the Hopf bifurcation in appendix.
	}
\end{abstract}

\end{frontmatter}
	
	\tableofcontents
	
	\section{Introduction}
	A large class of phenomena can be described using non-linear dynamical, in such distinct domains as economics, physics or biology. Investigating the impact of noisy perturbations on such systems is of great interest, and is currently an active field of research. In the last decades, it has been shown in many different areas that applying noise on a non-linear system can lead to many counter-intuitive phenomena, such as noise-induced stabilization~\cite{arnold:83}, oscillations~\cite{osc}, synchronization~\cite{pakdaman2004noise, hasegawa2008stochastic} or transport~\cite{ratchet, transport} and stochastic resonance~\cite{stocres, coherence}. From a mathematical perspective, understanding the interplay between non-linearities and noisy perturbations is a great challenge, with many applications, such as the design of early-warning signals predicting critical transitions~\cite{kuehn2011mathematical}. Several tools have been introduced, ranging from Lyapunov functionals, stationary solutions of Fokker-Planck equations and Feller scale functions (see e.g.~\cite{khasminskii:80,cherny2005singular} for an extensive account on these concepts), the theory of random dynamical systems (RDS)~\cite{arnold:98, crauel:99,arnold2001recent,baxendale1994stochastic,baxendale1999stability,keller1999numerical}, stochastic bifurcations and normal forms~\cite{namachchivaya1990stochastic} the study of moment equations \cite{tanabe:01,touboul-hermann-faugeras:11} or multiscale stochastic methods for slow-fast systems \cite{berglund:03,berglund-gentz:06, pavliotis:08}. 

	We focus here on the dynamics of stochastic differential equations (SDE) in which both the drift and the diffusion functions vanish at some particular absorbing points, also called singular points. Such SDEs arise in several applications including diffusion approximations of population models in ecology~\cite{cattiaux-collet-etal:09,lai2011effects}, neurosciences~\cite{pakdaman:10}, and mathematical finance~\cite{emanuel1982further,cox1975notes}. Beyond the classical multiplicative noise case (geometric Brownian motion), a number of state-dependent diffusion functions arise involving generally power functions. The question we address is how the interplay between the drift and the shape of the diffusion function affects the behavior of the system.  In particular, such systems present an absorbing state, whose qualitative properties will be shown to tightly depend on the local behavior of the drift and diffusion functions around this equilibrium. Moreover, the choice of a vanishing diffusion coefficient provides a framework to study subtle competitions between the drift and the noise, leading to a rich and generic phenomenology. 
	
	In detail, we analyze a general one dimensional stochastic differential equation
	\begin{equation*}
		dx_t=f(x_t)\,dt + \gamma(x_t)\,dW_t
	\end{equation*} 
	where $W=\big(W_t\big)_{t\geq 0}$ is a Brownian motion defined on a complete probability space $(\Omega, \mathcal{F},\mathbb{P})$ endowed with the natural filtration $\big(\mathcal{F}_t\big)_t$ of $W$. {We are interested in stationary solutions of this equation, defined as the solutions whose probability distribution is invariant in time. }
We assume that there exists a singular point $x^*$ such that $f(x^*)=\gamma(x^*)=0$. The system has a trivial solution distributed as a Dirac measure localized at the singular point $x^*$, the absorbing state (or absorbing boundary). We shall investigate the stochastic stability of this absorbing state, probabilities of absorption and the behavior of the system prior from absorption (or conditioned on not hitting the absorbing boundary). {Possible stationary distributions of the process conditioned on not hitting the absorbing boundary are called \emph{quasi-stationary solutions} (see Definition~\ref{def:QSD} below for a more rigorous definition, or~\cite{cattiaux-collet-etal:09} and references therein)}. In contrast to the case of ODEs, there are at least three different notions of stochastic stability: 
	\begin{enumerate}
		\item The \emph{almost sure exponential stability} defined by the property 
		\[\mathop{\textrm{limsup}}\limits_{t\to \infty} \quad \frac 1 t \log \vert x_t -x^*\vert <0 \quad a.s.\] 
		for any initial condition $x_0\in \R$. The \emph{almost sure exponentially instability} defined by the property 
		\[\mathop{\textrm{liminf}}\limits_{t\to \infty} \quad \frac 1 t \log \vert x_t -x^*\vert >0 \quad a.s.\]
		for any initial condition $x_0\in \R$ (see e.g. \cite{mao:08}). This definition implies a strong notion of stability close to the usual exponential stability property of deterministic dynamical systems: almost all trajectories exponentially fast converge to or diverge from $x^*$. In our stochastic context, this notion can appear quite restrictive, and we will hence also discuss a weaker notion of stability:
		\item The \emph{stochastic stability} (or stability in probability), defined by the property that for all $\mu \in (0,1)$ and $r>0$ there exists $\delta$ depending on $\mu$ and $r$ such that 
		\[\mathbb{P}(\{\vert x_t-x^*\vert <r \;\;\forall t\geq 0\}) \geq 1-\mu\] 
		whenever $\vert x_0-x^*\vert <\delta$ and otherwise it is said to be unstable in probability. The solution $x^*$ is \emph{stochastically asymptotically stable} (or asymptotically stable in probability)  if it is stable in probability and for every $\mu \in (0,1)$ there exists $\eta_0$ depending on $\mu$ such that
		\[\mathbb{P}(\lim\limits_{t\to\infty} \vert x_t-x^*\vert = 0 )\geq 1-\mu\]
		whenever $\vert x_0-x^* \vert <\eta_0$.
		\item The \emph{stability (resp. instability) in the first approximation}  of a stationary solution $\bar{X}$~\cite[Chap. 7.1]{khasminskii:80}, notion widely used in the RDS theory~\cite{arnold:98}, consists in that the identically zero solution of the linearized equation around $\bar{X}$ is almost surely exponentially stable (resp. unstable). This criterion will be chiefly used when dealing with non-Dirac stationary solutions in the applications section. 
	\end{enumerate}

	To be more specific, we now consider that $x^*=0$ and assume that $f(x)$ is a smooth function, which is often the case in applications and a basis of the deterministic bifurcation theory. In detail, we assume that $f(x)=\lambda\,x+g(x)$ where $g(x)/x \to 0$ when $x\to 0$ (i.e. $f'(0)=\lambda$), and hence consider the following class of SDEs on $\R$:
	\begin{equation}\label{eq:generalg}
		dx_t=(\lambda x_t+g(x_t))\,dt + \gamma(x_t) dW_t
	\end{equation}
	with $\gamma(x)\sim \sigma \vert x \vert^{\alpha}$ at $x=0$. The properties of the absorbing state will be described as a function of the local behavior of $g$ and $\gamma$ around $0$. This analysis will allow for instance to have a better understanding of the local and global behavior of diffusion in double-well potentials as represented in figure~\ref{fig:Question}.
	
	\begin{figure}
		\centering
			\includegraphics[width=.6\textwidth]{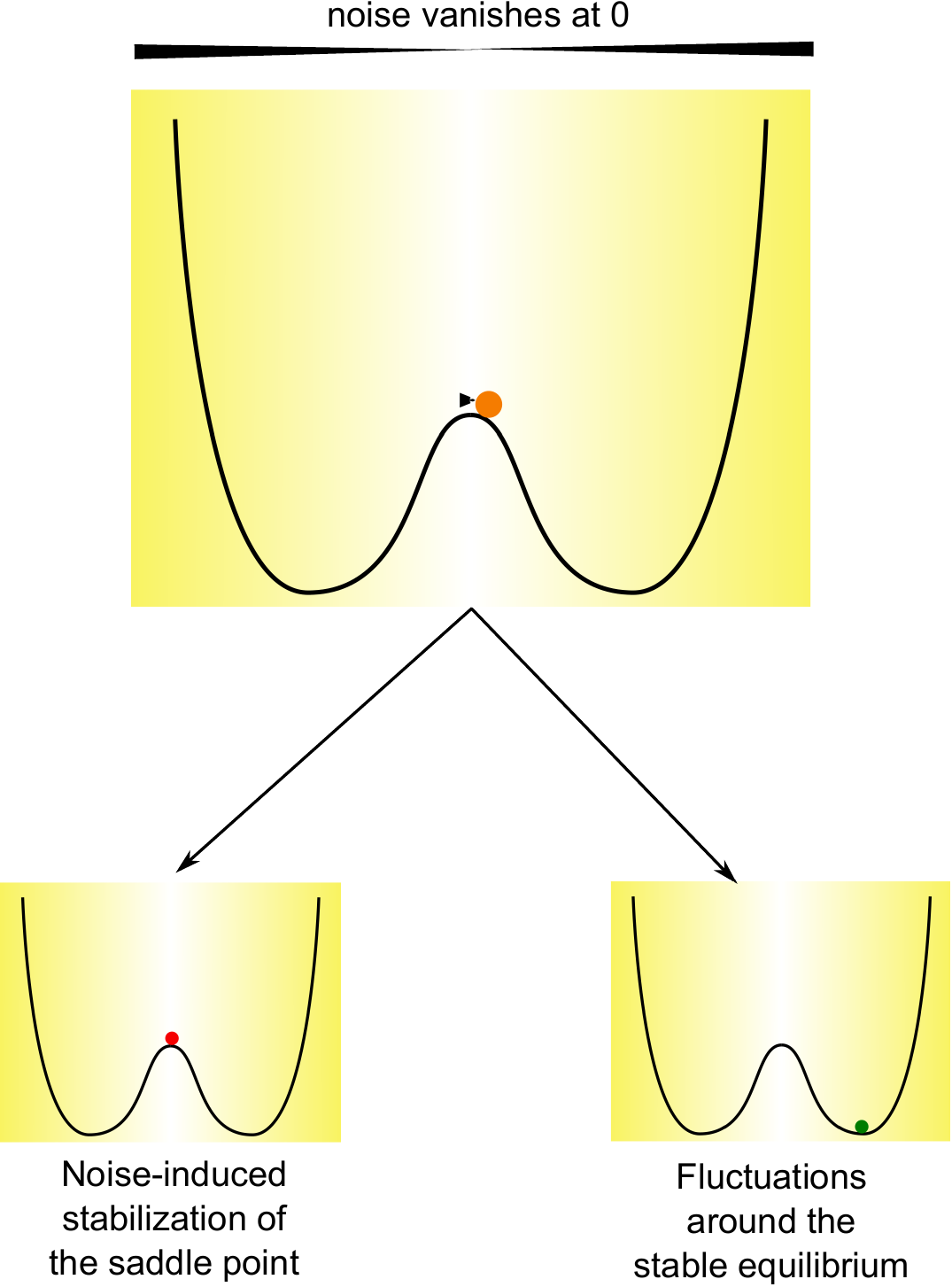}
		\caption{Stability of a singular saddle point in double-well potential perturbed by power diffusion coefficient. A particle located at the saddle point remains steady. If a small perturbation is applied, two scenarii are possible according to the local behavior of $\gamma(x)$ (see theorem~\ref{theo:Stability0}): either noise brings back the particle to the saddle point (left) or the particle falls into the well and fluctuates without ever reaching the saddle point (right). }
		\label{fig:Question}
	\end{figure}

	The behavior of the solutions to such one-dimensional singular SDEs has been intensively studied. In particular, the bifurcations of such systems with smooth diffusion functions was analyzed in the context of RDS theory in~\cite{crauel:99,arnold:98,baxendale1994stochastic}. These methods were developed in the context of Stratonovich SDEs (in which case the diffusion function appears in the effective drift in the It\^o formulation). The theory did not addressed cases where the diffusion is only H\"older continuous at the singular point. In another context, mainly using stochastic calculus theory~\cite{revuz-yor:99}, existence, uniqueness and type of singular points was address in~\cite{cherny2005singular}, and the study includes the case of power drift and diffusion functions. 
	
	In the present article, we shall consider both stability and absorption properties for a general class of singular SDEs, and aim at characterizing precisely the absorption time and global behavior of the solutions. New mathematical results proved here include almost sure absorption at the singular points and large deviations estimates of the absorption time, characterizations of possible stationary distributions, existence or non-existence of quasi-stationary distributions, and dependence of the stability of singular points on the \emph{local} H\"older coefficient of the diffusion function. The methodology is also extended to the case of the generic saddle-node bifurcation for which we completely describe the behavior as a function of drift and diffusion parameters. 

	The paper is organized as follows. We start by analyzing in detail in section~\ref{sec:general} the behavior of the solution of a general diffusion equation~\eqref{eq:generalg}. We investigate the existence, uniqueness and blow up of the solutions. The stochastic stability of the singular point is characterized and the absorption time at the singular point is shown to be almost surely finite if and only if $\alpha<1$. In that case, we show that quasi-stationary distributions exist when $\alpha< \frac 3 4$. In the case $\alpha\geq 1$, we study the existence of other stationary distributions. Application and extensions to the stochastic pitchfork and saddle-node bifurcations is discussed in section~\ref{sec:applications}, and universality properties of these behaviors is characterized.

	\section{General Theory}\label{sec:general}
	In this section we analyze the behavior of the solutions of a general equation~\eqref{eq:generalg} under a few assumptions on the function $g$ and $\gamma$. {We assume that $g$ is locally Lipschitz continuous and $\gamma$ satisfies, for some Borel function $\rho:\R_+\mapsto \R^+$ with $\int_{0^+} dx/\rho(x) = \infty$, that $\vert \gamma(x)-\gamma(y)\vert^2 \leq \rho(\vert x-y\vert)$. This imposes in particular the fact that $\gamma$ is locally H\"older $1/2$. We moreover make the following assumptions on the behavior of $g$ and $\gamma$ at $0$ and $\infty$:}
	\renewcommand{\theenumi}{(H\arabic{enumi})}
	\begin{enumerate}
		\item\label{assumption:H1} $g(x)\sim \mu x^{1+\kappa}$ at $x=0$, for some $\mu\in\R$ and $\kappa>0$
		\item\label{assumption:H2} $g(x)\leq -\nu x^{1+\beta}$ when $\vert x \vert \geq A$ for some $A>0$, for $\nu \in \R$ and $\beta>0$.
		\item\label{assumption:H3} $\gamma(x)\sim \sigma  x^{\alpha}$ at $x=0$, for some $\sigma\geq 0$ and $\alpha> 0$
		\item\label{assumption:H4} {At infinity, $\gamma$ is such that
		\begin{itemize}
			\item if $\nu > 0$, $\gamma^2(x) \leq d\;\vert x\vert^{1+\beta}$ for some $d>0$
			\item if $\nu<0$, $\gamma^2(x)\geq d\, \vert x\vert^{\delta}$ for some $d>0$ and $\delta>2+\beta$.
		\end{itemize}}
		\item\label{assumption:H5} $\gamma(x)$ only vanishes at $x=0$
	\end{enumerate}
	\begin{remark}
		In the above notations, for $x< 0$, $x^a$ shall be understood as $-\vert x \vert^a$. Away from the singular point, we shall assume that the maps $f$ and $g$ are smooth (at least Lipchitz continuous). 
	\end{remark}

	Assumptions~\ref{assumption:H1} and~\ref{assumption:H3} describe the behavior of the solutions close from zero and will govern the stability of the singular point $0$. The assumptions~\ref{assumption:H2},~\ref{assumption:H4} and~\ref{assumption:H5} are global properties that will be chiefly used for explosion matters, and existence of stationary distributions. Such assumptions are generally not necessary in the deterministic case: indeed, in the local unfolding of a bifurcation, the presence of a non-zero fixed point implies the local existence of a manifold of fixed points under some non-degeneracy conditions. However, in the stochastic case, as soon as the singular point (where noise vanishes) is unstable, the system will visit very regions very remote from the singular point with positive probability, and these latter conditions will come into play. {We observe that depending on the sign of $\nu$, we are making different assumptions on $\gamma$. When $\nu>0$, the drift has a contracting effect and prevents from blowing up, allowing a large class of diffusion coefficients provided that they do not diverge too fast. In contrast, when $\nu<0$, the solutions blow up in finite time in the absence of noise. In that case, a stringent assumption on the divergence of $\gamma$ is necessary (see Proposition~\ref{pro:ExistenceUniqueness}).}

	The stochastic process distributed as a Dirac at zero ($X_t=0$ a.s.) is solution of the equation whatever the parameters. Studying the stationary solutions and their stability will make essential use of the differential operator $\mathcal{L}$ associated with It\^o's representation of the diffusion acting on twice differentiable functions $V \in C^{2}(\R,\R)$:
	\[\Lop V(x) = (\lambda x +g(x))\,V'(x)+\frac 1 2 \gamma(x)^2 V''(x).\]

	\subsection{Existence and uniqueness of solutions}
	\label{sec:existence}
	The first question arising with equation \eqref{eq:generalg} is its well-posedness in this context where the diffusion coefficient is singular: for $\alpha < 1$, the diffusion coefficient is not Lipschitz continuous at $0$, and for $\alpha > 1$ it does not satisfy the linear growth condition generally imposed for the existence and uniqueness. Based on existing results~\cite{revuz-yor:99}, we prove the following property:
	\begin{proposition}\label{pro:ExistenceUniqueness}
		For any $\alpha \in [\frac 1 2, \infty)$, there exists a unique strong solution to the equation \eqref{eq:generalg}. Under the assumptions~\ref{assumption:H2} and~\ref{assumption:H4}, the solution never blows up in finite time.
	\end{proposition}
	\begin{proof}
	The proof uses classical results of stochastic analysis that can be found in~\cite{revuz-yor:99}. We provide a short description of the proof for the sake of completeness. 
	
	For $\alpha\geq 1$ is very classical. First, both the drift and diffusion functions are locally Lipschitz continuous, implying strong uniqueness of solutions (see e.g.~\cite[Chap. IX.2]{revuz-yor:99}). Second, the drift and the diffusion functions are locally bounded, ensuring the existence of solutions up to an explosion time.  

	For $\alpha <1$, the linear growth condition is clearly satisfied, but the diffusion coefficient is not Lipschitz at $0$. Existence and uniqueness was proved in~\cite[theorem IX.3.5]{revuz-yor:99} based on the properties of the local time at zero of the diffusion for $\alpha\geq \frac 1 2$ and non-uniqueness of solutions for $\alpha<\frac 1 2$. 

		Explosion properties are investigate through Feller's test (see~\cite[Proposition 5.22.]{karatzas-shreve:87}). The function $\frac{\lambda x + g(x)+1}{\gamma(x)^2}$ is locally integrable on both the intervals $(-\infty, 0)$ and $(0,\infty)$. We consider the interval $(0,\infty)$ for instance (the same principle applies for the interval $(-\infty,0)$), and for some $c>0$, we define the scale function:
		\[p(x):=\int_c^x \exp\left(-2\int_c^y \frac{\lambda \xi+g(\xi)}{\gamma^2(\xi)}\,d\xi \right)dy.\]
		The process almost surely never diverges before possibly reaching zero as soon as $p(x)$ is diverges at $\infty$. {At infinity, we have $\lambda \xi+g(\xi)\sim g(\xi)$ and:
		\begin{itemize}
			\item for $\nu>0$, we have $-2 \frac{g(\xi)}{\gamma^2(\xi)} \geq \frac{2\nu}{d}.$
			\item for $\nu<0$, we have $-2 \int_c^y \frac{g(\xi)}{\gamma^2(\xi)} \geq  \frac{2\nu}{d(2+\beta-\delta)} (y^{2+\beta-\delta} - c^{2+\beta-\delta}).$
		\end{itemize}
		In both cases, assumption~\ref{assumption:H4} implies that the integrand is lower bounded at infinity, hence $p(x)\to\infty$ when $x\to\infty$. We thus conclude that the process does not blow up in finite time in these cases, which ends the proof. }
	\end{proof}

	\subsection{Stability of $\delta_0$}
	We now address the stability of the equilibrium $\delta_0$, and prove the following:

	\begin{thm}\label{theo:Stability0}
		Let us denote by $\delta_0$  the stationary solution $X_t=0$ a.s. for all $t$. We have:
		\begin{itemize}
			\item For $\alpha<1$, $\delta_0$ is asymptotically stable in probability whatever $\lambda\in\R$ and $\sigma\neq0$.
				\item For $\alpha=1$, $\delta_0$ is:
				\renewcommand{\theenumi}{(\roman{enumi})}
				\begin{enumerate}
					\item asymptotically stable \emph{in probability} for $\lambda < \frac{\sigma^2}{2}$. Moreover, if $g(x)=-\mu x^{1+\kappa}$ for $\mu,\kappa>0$ and $\gamma(x)=\sigma\,x$, the stationary solution $\delta_0$ is \emph{almost surely exponentially stable} for $\lambda < \frac{\sigma^2}{2}$.

					\item unstable in probability for $\lambda > \frac{\sigma^2}{2}$ 
				\end{enumerate} 
			\item For $\alpha>1$, $\delta_0$ is :
			\begin{itemize}
				\item stable in probability for $\lambda<0$
				\item unstable in probability for $\lambda>0$
			\end{itemize}
		\end{itemize}
	\end{thm}
{
\begin{remark}
	 The case $\alpha \geq 1$ was investigated using tools from RDS theory in the Stratonovich formalism and speed functions in~\cite{arnold:98,crauel:99}. Our theorems extends this theory to the case $\alpha<1$ and uses Lyapunov functionals and stability theory for stochastic equations allowing to quantify the exponential speed of convergence.
\end{remark}
}
	\medskip

	\begin{proof}
		The case $\alpha=1$ was already understood in the context of stochastic bifurcations, and proofs of the result can be found in~\cite{arnold:83,mao:08}. The cases $\alpha\neq 1$ are simple applications of Lyapunov functional theory. The proof proceeds as follows:\\
		{\bf Case $\alpha<1$}: Let us consider $V(x)=x^{\alpha}$ as a Lyapunov functional for the dynamics. This function is clearly $C_2^0(\R)$, i.e. it is twice continuously differentiable except at $x=0$. Moreover, we have:
		\[\Lop V(x) = \alpha x^{\alpha-1}(\lambda x +g(x)) - \frac{\alpha(1-\alpha)}{2}\;\gamma(x)^2 x^{\alpha -2}\]
		Since we assume that $\alpha<1$ and $g(x)/x \to 0$ at $0$, the leading term close to $0$ is of order $- \alpha(1-\alpha)/2\;\sigma^2 x^{3\alpha -2}$ which is strictly negative for sufficiently small $x$. More precisely, there exists $r>0$ such that $V\in C_2^0(U_r)$ where $U_r$ is the open ball of radius $r$ and such that $\Lop V <0 $ for all $x \in U_r$. Moreover, for any $\varepsilon$ such that $0<\varepsilon<r$ we have $V(r)> 0$ and $\Lop V<-c_{\varepsilon}<0$ for all $r>x>\varepsilon$. Theorem V.4.1 of~\cite{khasminskii:80} therefore applies and concludes the proof of the stability in probability of the solution $\delta_0$ whatever the parameters $\lambda \in \R$, $\sigma>0$ and $\alpha \in [1/2,1)$.

		{\bf Case $\alpha=1$}: We start considering the case $\lambda<\sigma^2/2$ and define $\eta=\frac 1 2 -\frac{\lambda}{\sigma^2}>0$. The Lyapunov function $V:x\mapsto x^{\eta}$ satisfies:
		 \[\Lop V\;(x) = \eta x^{\eta}(\lambda +\frac 1 2 (\eta-1) \gamma^2(x))+\eta x^{\eta-1}\,g(x).\] 
		Under assumption~\ref{assumption:H1} we have $x\,g(x)\sim \mu x^{2+\kappa}$, hence is equivalent, close from zero, to 
		\[\eta x^{\eta}\left (\lambda +\frac {(\eta-1)}{2} \frac{\gamma^2(x)}{x^2}\right) = \frac 1 2 (\lambda-\frac{\sigma^2}{2})\eta x^{\eta}.\]
		 Since $(\lambda-\sigma^2/2)<0$, the same argument as in the case $\alpha<1$ ensures asymptotic stability of $\delta_0$ in probability for $\lambda<\sigma^2/2$.
	 If we now further assume $g(x)=-\mu x^{1+\kappa}$ and $\gamma(x)=\sigma\,x$, considering $V(x)=x^2$, we have:
		\begin{itemize}
			\item $\Lop V\;(x) = (2\lambda + \sigma^2) x^2 - 2\,\mu\,x^{2+\kappa} \leq (2\lambda + \sigma^2) V(x)$,
			\item $\vert V'(x)\sigma\;x\vert^2 = \vert 2\sigma x^{2} \vert^2 = 4 \sigma^2 V(x)$
		\end{itemize}
		Using the stochastic stability theorem \cite[Theorem 3.3]{mao:08}, we deduce that:
		\[\mathop{\textrm{limsup}}\limits_{t\to \infty} \frac 1 t \log \vert X(t) \vert \leq - \frac{4 \sigma^2-2(2\lambda + \sigma^2)}{4} = \lambda -\frac{\sigma^2}{2} \qquad \textrm{almost surely}.\]
		By definition of the almost sure exponential stability, we conclude that the solution $0$ is almost surely exponentially stable for any $\lambda < \frac{\sigma^2}{2}$.

		Let us now deal with the case $\lambda>\sigma^2/2$. We define $V(x)=-\log(x)$. This function tends to infinity as $x\to 0$ and belongs to $C_2^0(\R)$. Moreover,
		\[\Lop V(x) = (-\lambda +\frac{\gamma(x)^2}{2x})+ \frac{g(x)}{x}.\]
		which is equivalent at zero to $-\lambda + \sigma^2/2 <0$. Hence there exists $r>0$ such that $\Lop V \leq 0$ for any $x<r$, and moreover we have for any $\varepsilon>0$ and $r>x>\varepsilon$, $V(x)\geq 0$ and $\Lop V (x) <-c_{\varepsilon}<0$. Theorem V.4.2 of~\cite{khasminskii:80} applies and ensures that the solution $\delta_0$ is not stable in probability.

		{\bf Case $\alpha>1$:} For $V(x)=x^2$ around zero, we have:
		\[\Lop V (x) = 2\lambda x^2 +2x g(x)+ \gamma(x)^2\]
		and under the assumption that $g(x)\sim \mu x^{1+\kappa}$ at zero and $\alpha>1$, the term $\Lop V$ is equivalent to $2 \lambda x^2$ and hence locally has the sign of $\lambda$. For $\lambda<0$, we can directly apply theorem V.4.1 of~\cite{khasminskii:80}. For $\lambda>0$, we use $V(x)=-\log(x)$ again and conclude that $0$ is not stable in probability. 
	\end{proof}

	\subsection{Absorption time}
	We know that $X_t=0$ a.s. is a stationary solution of the SDE. It is an absorbing state of the diffusion, i.e. as soon as the solution reaches the singular point, the solution is almost surely identically stuck at the singular point for the whole subsequent evolution. We are now interested in the absorption probability, namely the probability of reaching zero in finite time. We distinguish between the cases $\alpha<1$ and $\alpha\geq 1$. 

	\subsubsection{Case $\alpha<1$}
	\begin{proposition}\label{pro:firsthittingtimezero}
		Under assumptions~\ref{assumption:H1}-\ref{assumption:H5}, the first hitting time $\tau_0$ of zero of the solution of
		\[dx_t=(\lambda x +g(x) )\,dt + \gamma(x)dW_t\]
		 is almost surely finite for any $\alpha<1$ and $\sigma>0$. This is also the case for $\nu=0$ and $\lambda<0$. 
	\end{proposition}

	\begin{proof}
		In order to establish this property, we use a sufficient condition derived from Feller's test for expositions and evidence the almost sure finite exit time of the interval $]0,\infty[$. The finiteness of the exit time of $]0,\infty[$ will readily imply the finiteness of the hitting time of $0$, since the process is almost surely bounded for all finite time  as proved in proposition~\ref{pro:ExistenceUniqueness} under the assumptions of the proposition. 

		Let us define the interval $I=(0,\infty)$, and let $c\in I$ arbitrarily chosen. Under assumption~\ref{assumption:H5}, we have that for all $x\in I$, the diffusion coefficient is always strictly positive, hence the function $(1+\vert \lambda x +g(x)\vert )/\gamma(x)^2$ is locally integrable on $I$. The finiteness of the exit time of the process from the interval $I$ can be established through the analysis of Feller's scale functions. For convenience, we introduce the function:
		\[G(x)=\int_c^x \frac{(\lambda \xi +g(\xi))}{\gamma^{2}(\xi)}d\xi.\]
		Feller's scale functions written as a function of $G$ read:
		\[\begin{cases}
			p(x) &\stackrel{def}{=} \int_c^x \exp\left\{-2G(y)\right\}\,dy\\
			m(dx) & \stackrel{def}{=} \frac{2\,\exp(2\,G(x))}{\gamma^{2}(x)}\,dx\\
			v(x) &\stackrel{def}{=} \int_c^x (p(x)-p(y))\,m(dy)
		\end{cases}\]

		We prove that  (a) $\lim_{x\to\infty} p(x)=\infty$ and (b) $\lim_{x\to 0^+} v(x) < \infty$. The property (a) has been addressed in the proof of proposition~\ref{pro:ExistenceUniqueness}. 
		We hence only need to prove the property (b) which is slightly more subtle. The function $v(x)$ is indeed defined as an integral, which is well behaved except possibly at $x=0$. Close to this point, we have:
		\begin{align*}
			G(y)-G(z)&=\int_{z}^y \frac{\lambda \xi +g(\xi)}{\gamma^{2}(\xi)}d\xi\\
			&\sim_{0^+}\frac{\lambda}{2\sigma^2 (1-\alpha)} (y^{2(1-\alpha)}-z^{2(1-\alpha)})+\frac{\mu}{\sigma^2\phi} (y^{\phi}-z^{\phi})
		\end{align*}
		with $\phi=2+\kappa-2\alpha$ and this function is smooth and bounded close to zero. Let us denote by $M$ the supremum of $\vert G(y)-G(z)\vert$ on $[0,c]$, and by $F(x,y)$ the function defined for $0< y \leq x \leq c $:
		\[F(x,y)=\int_y^x \exp(2(G(y)-G(z)))\,dz \leq e^{2M}(x-y).\]
		It is then clear that:
		\begin{equation*}
			v(x)=2\int_{c}^{x} \frac{F(x,y)}{\gamma^2(y)} \,dy \leq 2\int_{c}^{x} \frac{e^{2M}(x-y)}{\gamma^2(y)}\,dy
		\end{equation*}
		which is integrable at zero since the integrand is of order $x^{1-2\alpha}$ and hence integrable at zero under the condition that $\alpha<1$. By application of Feller's test for explosion~\cite[Prop. 5.32 (iii)]{karatzas-shreve:87} we conclude that if $S$ denotes the exit time of the process from the interval $I$, we have $\mathbbm{P}(S<\infty)=1$. Moreover, we have seen in proposition~\ref{pro:ExistenceUniqueness} that under the assumptions of the proposition, the process is almost surely bounded. We hence conclude that necessarily the process reaches $0$ in finite time and is absorbed (this is also a direct consequence of Feller's test for explosion~\cite[Prop. 5.22]{karatzas-shreve:87} under the condition that $p(x)\to \infty$ when $x\to\infty$ if $p(x)>-\infty$ when $x\to 0^+$, which is a simple consequence of the above estimates).

		Let us now deal with the case $\nu=0$ and $\lambda<0$, we have $G(x)=\lambda x^{2(1-\alpha)} + C$ for some constant $C$. The same reasoning as above ensures that $p(x)$ tends to infinity as $x\to\infty$. Demonstrating that $v(x)$ is bounded at zero did not use the fact that $\nu\neq 0$ and hence applies in the present case. We conclude that $P(S<\infty)=1$, hence the process almost surely reaches $0$ in finite time. 

	\end{proof}

	We hence proved that for $\alpha<1$, the solutions almost surely hit zero in finite time, where they stay trapped for the subsequent dynamics, as soon as the diffusion coefficient $\sigma$ is nonzero. 

	When $\sigma=0$, this is never the case: in the case $\lambda<0$, the solutions of the related ODE converge exponentially fast towards zero but never actually reach it, and in the case $\lambda>0$, the solutions escape from zero exponentially. It is hence of interest to analyze the behavior of the system in the small noise limit. We focus on the case $\lambda>0$ (in which case $0$ is unstable in the noiseless system and a.s. reached in finite time in the stochastic case) and use Freidlin and Wentzell large deviations theory~\cite{freidlin-wentzell:98} to analyze in detail the behavior of the system in the small noise limit. In order to obtain precise quantitative estimates, we will assume that $g(x)=-\mu x^{1+\kappa}$ with $\kappa>0$ and $\mu>0$ and $\gamma(x)=\sigma \vert x \vert^{\alpha}$, corresponding to a double-well potential. Note that the method developed here does not depend on the choice of $g(x)$. 

	\begin{remark}
		The Freidlin and Wentzell theory~\cite{freidlin-wentzell:98} was developed for constant diffusion coefficients. Extensive works have been performed to generalize this result, and for our purposes, we may cite~\cite{puhalskii:04} who showed the validity of Freidlin and Wentzell estimates for continuous and unbounded coefficients with possibly singular diffusion coefficients, assuming that the drift $\lambda x + g(x)$ and $\gamma(x)^2$ Lipschitz-continuous, which is the case under our current assumptions. 
	\end{remark}

	We analyze the properties of the absoption time of the system with initial condition chosen as the fixed point of the noiseless system:
	\begin{equation*}
	\tau_0:=\inf\left\{t>0;\ X_t=0 \quad \vert \quad X_0=\left(\frac{\lambda}{\mu}\right)^{1/\kappa}\right\}
	\end{equation*}
	In the case $\alpha<1$, one can expect from the large deviation theory that $\mathbbm{E}\left[\tau_0\right]$ becomes exponentially large as $\sigma\to 0$, i.e behaves in the leading order as $e^{V/\sigma^2}$ where $V$ is a positive constant. More precisely, we have:
	\begin{proposition}\label{Prop:largedeviations}
	For all $\frac{1}{2}<\alpha<1$ and $\lambda>0$:
	\begin{equation*}
	\displaystyle{\lim_{\sigma \to 0}}\;\;\sigma^2 \log \mathbbm{E}\left[\tau_0\right] = \lambda \left(\frac{\lambda}{\mu}\right)^{2(1-\alpha)/\kappa}\left[\frac 1 {2(1-\alpha)} - \frac 1 {2(1-\alpha)+\kappa}\right] >0
	\end{equation*}
	\end{proposition}
	\begin{proof}
	By application \cite[Theorem 5.7.11.]{dembo-zeitouni:10}, the problem is reduced to compute at the boundary $x=0$ the quasipotential $U(x)$, solution of the Hamilton-Jacobi equation:
	\begin{equation}
	(\lambda x-\mu x^{1+\kappa})U'(x)+\frac{x^{2\alpha}}{2}(U'(x))^2 = 0
	\end{equation}
	with boundary condition $U((\frac{\lambda}{\mu})^{\frac 1 \kappa})=0$. That is, one has to compute:
	\begin{equation}
	U(x)=\int_x^{\left(\frac{\lambda}{\mu}\right)^{1/\kappa}} \frac{\lambda u-\mu u^{1+\kappa}}{u^{2\alpha}}du
	\end{equation}
	Simple algebra yields the expression:
	\begin{equation}
	U(0)=\lambda \left(\frac{\lambda}{\mu}\right)^{2(1-\alpha)/\kappa}\left[\frac 1 {2(1-\alpha)} - \frac 1 {2(1-\alpha)+\kappa}\right]
	\end{equation}
	\end{proof}

	As a consequence, for $\alpha<1$ fixed, the aborption time becomes exponentially large, which explains the transition between the regime $\sigma>0$ (where the system a.s. hits zero) and $\sigma=0$. A critical boundary of this domain is the point $\alpha=1$: in that case, the system a.s. never hits zero in finite time as shown in section~\ref{sec:NotReachZero}.

	In particular, proposition~\ref{Prop:largedeviations} ensures that the convergence as $\alpha\to 1^-$ with $\lambda>0$ fixed displays a super-exponential time behavior:
	$$\displaystyle{\lim_{\alpha \to 1^-} \lim_{\sigma \to 0}} \quad  \sigma^2 \log\mathbbm{E}\left[\tau_0\right] = \infty $$

	Going deeper into the analysis of the above convergence allows to understand more precisely what happens in the critical parameter region where both $\alpha \to 1^-$ and $\lambda \to 0^+$, in the limit of small noise $\sigma\to 0$. Such a ``blow up'' reveals the following picture around the point $(\alpha,\lambda,\sigma)=(1,0,0)$ in parameter space:

	\begin{corollary}\label{Prop:blowup}
	Assume that $\lambda \to 0^+$ and $\alpha \to 1^-$ such that 
	\begin{equation}\label{constraint}
	\frac{\lambda}{(1-\alpha)} \to c \in [0,\infty]
	\end{equation}
	Then, one has the following three regimes:
	\begin{itemize}
	\item if $c=0$, the absorption time is sub-exponential : $$\displaystyle{\lim_{\alpha \to 1,\lambda \to 0} \lim_{\sigma \to 0}} \quad \sigma^2 \log\mathbbm{E}\left[\tau_0\right] =0$$
	\item if $0<c<\infty$, the absorption time is exponential : $$\displaystyle{\lim_{\alpha \to 1,\lambda \to 0} \lim_{\sigma \to 0}} \quad  \sigma^2 \log\mathbbm{E}\left[\tau_0\right] =\frac{c}{{2\mu^{2(1-\alpha)/\kappa}}}$$
	\item if $c=\infty$, the absorption time is super-exponential : $$\displaystyle{\lim_{\alpha \to 1,\lambda \to 0} \lim_{\sigma \to 0}} \quad  \sigma^2 \log\mathbbm{E}\left[\tau_0\right] = \infty $$
	\end{itemize}
	where $\displaystyle{\lim_{\alpha \to 1,\lambda \to 0}}$ is constrained by Eq. (\ref{constraint}).
	\end{corollary}
	\begin{proof} 
		This proposition is a consequence of the analysis of the different asymptotic regimes of the critical exponent $c^*$ obtained in proposition~\ref{Prop:largedeviations}:
		\[c^*=\lambda \left(\frac{\lambda}{\mu}\right)^{2(1-\alpha)/\kappa}\left[\frac 1 {2(1-\alpha)} - \frac 1 {2(1-\alpha)+\kappa}\right].\]
		The leading term in the considered limit is proportional to $\lambda^{1+2(1-\alpha)/\kappa}/(1-\alpha)$. A series expansion provides the following expression:
		\[c^*\sim \frac 1 {2\mu^{2(1-\alpha)/\kappa}} (\frac{\lambda}{\alpha-1} + \frac{2}{\kappa}\lambda \log(\lambda))+O(\lambda\log(\lambda)^2(1-\alpha)). \]
		The different regimes identified in the corollary immediately follow from the analysis of the limits of this exponent as $\lambda\to 0^+$ and $\alpha\to 1^-$ subject to the condition~\eqref{constraint}.
		\end{proof}
	\begin{remark}
		Note that the regimes do not depend on $\mu$ or $\kappa$ qualitatively and the property is purely local: as $\lambda$ goes to zero, the integration range becomes increasingly small. The property is hence valid for equation~\eqref{eq:generalg} with function $g(x)$ equivalent at zero to $-\mu x^{1+\kappa}$. However, the asymptotic exponential divergence rate quantitatively depends on the local behavior of the flow close to zero, as observed in the case $0<c<\infty$.
	\end{remark}

	From a dynamical systems viewpoint, this type of solutions corresponds to a relatively sharp convergence towards the equilibrium. Indeed, the dynamics stops at the almost sure hitting time of the singular point and the whole subsequent dynamics is deterministic and constant. This will not be the case for more regular diffusion coefficients ($\alpha\geq 1$) as we now show.

	\subsubsection{For $\alpha\geq 1$ the solution a.s. never reaches $0$}\label{sec:NotReachZero}

	We now deal with the case $\alpha\geq 1$ and show that in contrast to the case $\alpha<1$, even when $\delta_0$ is stable in probability or almost surely exponentially stable, the solutions of the SDE almost surely never reaches zero. 
	\begin{thm}\label{thm:NotTouchingZero}
		Let $\alpha\geq 1$. Then for any $x_0\in\R$ we have 
		\[\mathbbm{P}\big[\forall t\geq 0\;,\;\; x_t\neq 0 \big]=1\]
		That is, almost any sample path starting from a non-zero initial condition will never reach zero.
	\end{thm}
	\begin{proof}
		The result is already known for $\alpha=1$ (see e.g.~\cite{mao:08}). We extend the demonstration to our general case. Denote by $\tau$ the first hitting time of zero of the process with initial condition $x_0\neq 0$ at $t=0$. Assuming that the conclusion is false, we can find a pair of constants $T>0$ and $\theta>1$ sufficiently large such that $P(B)>0$ where $B$ is defined as:
		\[B=\{\tau\leq T \text{ and } \vert x(t) \vert \leq \theta-1 \text{ for all } t \leq \tau \}\]
		Moreover, under the assumption $\alpha \geq 1$, both the drift and diffusion functions are locally Lipschitz-continuous, hence there exists a positive constant $K_{\theta}$ such that:
		\[\vert \lambda x + g(x) \vert \vee \gamma(x) \leq K_{\theta} \vert x \vert \quad \text{ for all } \vert x\vert \leq \theta,\; t\leq T.\]
		Let now $V(x):=\vert x \vert^{-1}$. We have for $0<\vert x \vert \leq \theta $ and $ t \leq T$
		\begin{align*}
			\Lop V (x) &\leq K_{\theta}(1+K_{\theta}) V(x)
		\end{align*}
		Let us now define $\tau_{\varepsilon}=\inf\{t\geq 0;\; \vert x_t\vert \notin (\varepsilon, \theta)\}$. This is a stopping time, so by It\^o's formula and the optimal sampling theorem with the bounded stopping time $\tau_{\varepsilon} \wedge T$, we have:
		\begin{align*}
			&\Exp{[e^{- K_{\theta}(1+K_{\theta})(\tau_{\varepsilon} \wedge T)} V(x_{\tau_{\varepsilon \wedge T}})]} = V(x_0) +\\
			&\qquad  + \Exp{\Big[\int_{0}^{\tau_{\varepsilon} \wedge T} e^{- K_{\theta}(1+K_{\theta})\,s} [-K_{\theta}(1+K_{\theta})V(x_s)+(\Lop V)(x_s)]\,ds\Big]}\\
			&\leq \vert x_0\vert^{-1}
		\end{align*}
		We know that for any $\omega \in B$, the time $\tau_{\varepsilon}$ is smaller than $T$ and moreover that $x_{\tau_{\varepsilon}}=\varepsilon$ by definition of $B$, hence we have:
		\[\vert x_0 \vert^{-1}\geq \Exp{[e^{- K_{\theta}(1+K_{\theta})(\tau_{\varepsilon} \wedge T)} V(x_{\tau_{\varepsilon \wedge T}})]}  \geq \Exp{\Big[e^{- K_{\theta}(1+K_{\theta})T} \varepsilon^{-1}\mathbbm{1}_{B}\Big]}\]
		i.e. 
		\[\mathbbm{P}(B)\leq \varepsilon \vert x_0 \vert^{-1} e^{ K_{\theta}(1+K_{\theta})T}\]
		This inequality is true whatever $\varepsilon$ (we recall that the definition of $B$ did not involve $\varepsilon$), hence valid in particular when $\varepsilon$ is arbitrarily small, and necessarily $\mathbbm{P}(B)=0$ which contradicts the definition of $B$. 
	\end{proof}

	\subsection{Stationary and quasi-stationary solutions}
	We are now interested in identifying possible stationary solutions. We already mentioned that the null stochastic process (distributed as a Dirac delta measure at $0$) is always a stationary solution of the system. Moreover, for $\alpha<1$, we have shown that the process is almost surely absorbed at zero in finite time. This property prevents from the existence of any stationary solution, and as far as the permanent dynamics is concerned, it is perfectly characterized by the absorption property. In order to further describe the behavior of the solutions, one may wonder how the solutions that did not yet reach zero are distributed. Stationary solutions of the process conditioned on never hitting zero are called quasi-stationary distributions:
\begin{definition}\label{def:QSD}
	{A quasi-stationary distribution $p$ for the process $(x_t)_{t\geq 0}$ is a probability measure supported on $\mathcal{I}=(0,\infty)$ or $(-\infty,0)$ satisfying, for $x_0\stackrel{\textrm{law}}{=} p$ and all $t\geq 0$
	\[\mathbbm{P}(x_t \in A \; \big \vert \; \tau_0>t) = p(A) \qquad \forall A\in \textrm{Borel}\,(\mathcal{I})\]
	}
\end{definition} 
In the case $\gamma(x)=\sigma \vert x \vert^{\alpha}$ with $\alpha<3/4$, we can show that there exists a quasi-stationary distribution that attracts exponentially fast the conditioned process. In detail, we have:
	\begin{proposition}\label{pro:QSD}
	Let us assume that $\nu>0$, $g$ is differentiable, $g'$ is bounded at zero and smaller than $C x^{2\beta}$ for some $C\in\R$ at infinity, that assumptions~\ref{assumption:H1}-\ref{assumption:H5} hold and $\alpha\in[1/2,3/4)$. The process conditioned on never reaching zero and with a bounded initial distribution converges exponentially fast towards a probability measure $p_{\pm}$ for any positive (resp. negative) initial condition. These are called quasi-stationary solutions, or Yaglom limit. 	
	\end{proposition}
	\begin{proof}
	We apply a result due to Cattiaux and collaborators~\cite{cattiaux-collet-etal:09} in the domain of diffusion models in population dynamics. In their article, they consider a diffusion of type:
		\[dX_t=dB_t-q(X_t)\,dt\]
		under a few conditions on $q$. Our problem can be brought to their framework changing variables and considering: $X_t=x_t^{1-\alpha}/(\sigma(1-\alpha))$. It\^o's formula allows putting our equations in that form, with 
		\[q(x)= - \frac 1 \sigma x^{-\alpha} f(x) + \frac{\alpha}{2(1-\alpha)} \frac 1 x\]
		and where $f(x)=\lambda x +g(x)$ and $Z=({\sigma(1-\alpha)}{x})^{1/(1-\alpha)}$. Under the conditions that (i) the first hitting time of zero is almost surely finite and (ii) that the conditions: $\Delta=-\inf_{y\in (0,\infty)} q^2(y)-q'(y)<\infty$ and $\lim_{y\to\infty} q^2(y)-q'(y)=\infty$, they show that the spectrum of the Kolmogorov backward equation (for the related operator $\Lop$) has a purely discrete spectrum with no zero eigenvalue. Moreover, they show that a rescaled version of the related eigenfunction is integrable and exponentially attract all initial conditions under the conditions:
		\begin{equation}\label{eq:ConditionAmaury1}
			\int_0^1 \frac{1}{q(y)^2-q'(y)+\Delta+2} e^{-G(y)}\,dy<\infty
		\end{equation}
		with $G(y)=\int_1^x 2q(y)\,dy$,
		\begin{equation}\label{eq:ConditionAmaury2}
			\int_1^{\infty} e^{-G(y)} dy<\infty \; , \; \int_0^1 x e^{-G(x)/2} dx < \infty \; \text{and}\; \int_1^{\infty}e^{G(y)}\int_y^{\infty}e^{G(z)}\,dzdy<\infty
		\end{equation}

		In order to apply their result, we demonstrate that all these conditions are satisfied in our framework. First of all, under the condition $\alpha<1$ and $\nu>0$, we showed that the process $x_t$ solution of~\eqref{eq:generalg} almost surely reach zero in finite time (proposition~\ref{pro:firsthittingtimezero}). This is hence also the case of $X_t$ since $(1-\alpha)>0$ in that case. Moreover, the derivative of the function $q(x)$ reads:
		\[q'(x)=\alpha Z^{-1}f(Z)-f'(Z)-\frac {\alpha}{2(1-\alpha)}\frac 1 {x^2}.\]
		Around zero, because of our assumptions on $g$, we have: 
		\[q(x)\sim \frac{\alpha}{2(1-\alpha)}\frac 1 x \; \text {and}\; q'(x)\sim -\frac{\alpha}{2(1-\alpha)}\frac 1 {x^2}\]

		and hence 
		\[q^2(x)-q'(x)\sim_{x\to 0^+} \frac{\alpha(2-\alpha)}{4(1-\alpha)^2}\;\frac 1 {x^2}\]
		which tends to infinity as $x\to 0^+$.
		At $x = \infty$, we have:
		\[q^2(x)-q'(x) \geq \frac{\mu^2}{\sigma^2} (1-\varepsilon) Z^{2(1-\alpha)+2\beta}\]
		for any $\varepsilon>0$ and hence tends to infinity when $x\to \infty$. In particular, $\Delta=-\inf_{y\in (0,\infty)} q^2(y)-q'(y)<\infty$

		Let $G(x)=\int_1^x 2\,q(y)\,dy$. We now show that the condition~\eqref{eq:ConditionAmaury1} is satisfied. At infinity, $q(x)$ is greater or equal to $C' x^{1+\beta/(1-\alpha)}$ with $C'>0$. It is then easy to show that $\lim_{x\to\infty} \frac{G(x)}{x}=\infty$. Moreover, 
		we have
		\[\quad A:=\lim_{x\to 0^+}\Big(G(x)-\frac{\alpha}{1-\alpha}\log(x)\Big) \in (-\infty,\infty).\]
		Therefore, near zero, the integrand involved in the expression of $Q(x)$ is proportional to $y^{2-\frac{\alpha}{1-\alpha}}$. Under the condition $\alpha<\frac 3 4 $, condition~\eqref{eq:ConditionAmaury1} holds. 
		Showing that condition~\eqref{eq:ConditionAmaury2} is valid is straightforward at this point: $G$ grows at least linearly at $\infty$ ensuring that $\int_1^{\infty} e^{-G(y)} dy<\infty$. Moreover, around zero, the function $x e^{-G(x)/2}$ is of order $x^{1-\frac 1 2 \frac{\alpha}{1-\alpha}}$ which is integrable at zero for $\alpha<\frac 4 5$, hence in particular for $\alpha<\frac 3 4 $. The main results of~\cite{cattiaux-collet-etal:09} (namely their theorem 5.2, proposition 5.5, corollaries 6.1 and 6.2) apply, and ensure the existence and uniqueness of a quasi-stationary distribution attracting exponentially fast the law of the diffusion conditioned on not hitting zero with initial laws having bounded support. 
	\end{proof}

\begin{remark}
	The above proposition ensures existence of a quasi-stationary solution for $\alpha<3/4$, using a sufficient condition for existence of such measures, hence does not preclude existence of quasi-stationary distributions for $\alpha\in[3/4,1)$. Numerical explorations of such system tend to show that the bound $\alpha=3/4$ is actually optimal. For instance, Figure~\ref{fig:QSD} shows that the solutions to the stochastic subcritical pitchfork equation:
	\begin{equation}\label{eq:subpitch}
		dx_t=(\lambda x-x^3)\,dt+\sigma x_t^{\alpha}dW_t
	\end{equation}
	do present invariant distributions of the process conditioned on not hitting zero for $\alpha<3/4$, but for $\alpha>3/4$, the distribution does not reach a stationary state but tends to approach the stable solution $\delta_0$ (this is also what happens for $\alpha=1$ when no absorption is present but an a.s. exponential convergence occurs).
	
	\begin{figure}[h]
		\centering
			\subfigure[$\alpha=0.6$]{\includegraphics[width=.4\textwidth]{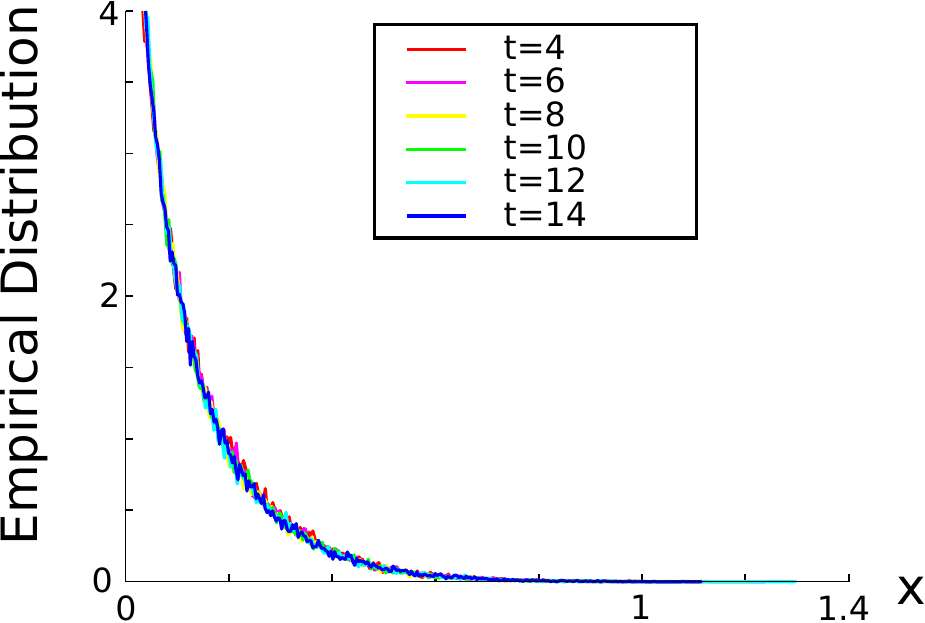}}\qquad
			\subfigure[$\alpha=0.85$]{\includegraphics[width=.4\textwidth]{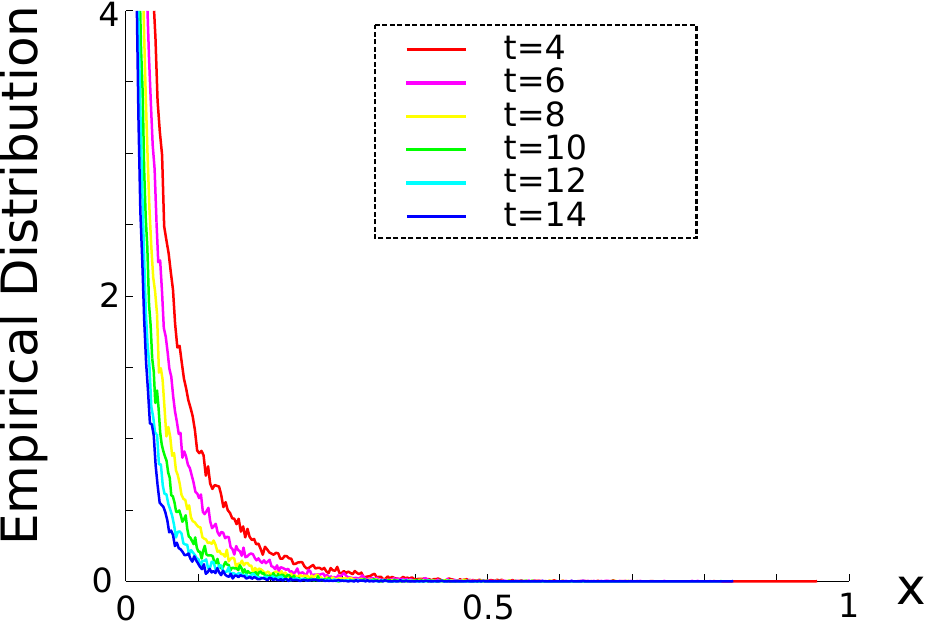}}
		\caption{Distribution (ordinate truncated for legibility)of the trajectories that did not reach zero in the subcritical pitchfork model~\eqref{eq:subpitch} with $\lambda=-0.5$, $\sigma=0.5$ different values of $\alpha$, computed at $6$ times. (a) $\alpha=0.6$: stabilization on a quasi-stationary distribution that do not depend on time, (b) $\alpha=0.85$: the repartition varies as a function of time and all trajectories approach zero, arguing for the non-existence of quasi-stationary distribution. Simulations were done using a particle method with $100\,000$ particles, using the Euler-Maruyama scheme with $dt=0.01$.}
		\label{fig:QSD}
	\end{figure}
	
\end{remark}

	Let us now deal with the general case with diffusion coefficient $\gamma(x)$ and $\alpha\geq 1$. In that case, the process is almost surely never absorbed at zero. Possible stationary solutions with smooth probability density function $p$ with respect to Lebesgue's measure satisfy the forward Kolmogorov (or Fokker-Planck) equation:
	\begin{equation}\label{eq:Kolmo}
		\der{}{x} \left( (\lambda\,x+g(x))p(x) - \frac 1 2 \der{}{x}\left(\gamma(x)^2 p(x)\right)\right) = 0.
	\end{equation}
	i.e.
	\[(\lambda\,x+g(x)-\gamma'(x)\gamma(x))p(x)-\frac{1}{2}\gamma(x)^{2}p'(x)=K\]
	for some $K\in\R$. If we can find solutions to these equations subject to the conditions $p(x)\geq 0$ for all $x\in\R$ and $\int_{\R}p(x)\,dx=1$ (\emph{acceptable solutions}), then these are probability density functions of stationary solutions of the diffusion equation. For simplicity, we consider that assumption~\ref{assumption:H5} is valid, i.e. $\gamma(x)$ only vanishes at $x=0$. In that case, since $\delta_0$ is a solution to the SDE~\eqref{eq:generalg}, stationary solutions will either have support on $\R^+$ or $\R^-$. In the case where the dynamics has additional singular points, the Kolmogorov equation will define solutions in each interval delimited by the singular points (see section~\ref{sec:SN}). The Kolmogorov equation is a linear ordinary differential equation with constant source term. The homogeneous solutions (with $K=0$) is given by:
	\[p_0(x)=C\gamma^{-2}(x)\exp\left(2\int_{1}^x (\frac{\lambda y}{\gamma(y)^2} + \frac{g(y)}{\gamma(y)^{2}})\,dy\right)\]
	defined for $x\neq 0$. Particular solutions are given by:
	\begin{multline}\label{eq:Solution1}
		p_K(x)=\gamma^{-2}(x)\left(K \int_1^x \exp\left(-2\int_1^y \frac{2\lambda z}{\gamma(z)^2} + 2\frac{g(z)}{\gamma(z)^{2}}\,dz\right)dy + C\right)\\
		\times \exp\left(\int_1^x (\frac{2\lambda y}{\gamma(y)^2} + 2\frac{g(y)}{\gamma(y)^{2}})\,dy\right)
	\end{multline}

	We have the following:
	\begin{proposition}\label{pro:Station}
		In addition to the stationary distribution $\delta_0$, we have:
		\begin{itemize}
			\item For $\alpha=1$ and:
			\begin{itemize}
				\item $\lambda<\sigma^2/2$ there exists no integrable solution of the Kolmogorov equation~\eqref{eq:Kolmo}
				\item {$\lambda> \sigma^2/2$ there exists two invariant probability distributions $Z_+ p_0(x)\ind{x>0}$ and $Z_-\,p_0(x)\ind{x<0}$ (where $Z^+$ and $Z^-$ generically denote normalization constants) under assumption~\ref{assumption:H4}.
				}
			\end{itemize}
			\item For $\alpha>1$, {the same results are found but the bifurcation arises at $\lambda=0$}
		\end{itemize}
	\end{proposition}
	\begin{remark}
	In the case $\alpha<1$, we already mentioned that no stationary solution can exist. We can show using the expressions of the solutions to the Kolmogorov equation that no solution is integrable.	
	\end{remark}
	\begin{proof}

		For $\alpha>1$, let us analyze the integrability of $p_0(x)$. Around $0$, the function $x^{-2\alpha}$ is not integrable. The leading term of the exponential around zero is given by $\frac{\lambda}{\sigma^2(1-\alpha)} x^{2(1-\alpha)}$, hence diverges at zero. If $\lambda<0$, this term tends to $+\infty$ and is not integrable, and if $\lambda>0$, the exponential term tends to zero and the function is hence integrable at $0$ and $p(0)=0$. For $x\to\infty$, we distinguish two cases:  
		\begin{itemize}
			\item {for $\nu>0$, the leading term of the exponential in the expression of $p_0(x)$ is given by $-2\nu y^{1+\beta}/\gamma^2(y) \leq -2\nu/d$, and therefore the map
			\[\varphi(x)=\exp\left(2\int_{1}^x \frac{\lambda y + g(y)}{\gamma^{2}(y)}\,dy\right)\]
			has a limit at infinity. For $A$ large enough so that $\lambda x + g(x)<-1$ for all $x\geq A$, we have:
			\begin{multline*}
				\int_A^x \frac{1}{\gamma^{2}(y)}\exp\left(2\int_1^y \frac{\lambda z +g(z)}{\gamma^{2}(z)} \,dz\right)\,dy \\
				\leq 
				\int_A^x -2\frac{\lambda y + g(y)}{\gamma^{2}(y)}\exp\left(2\int_1^y \frac{\lambda z +g(z)}{\gamma^{2}(z)}  \,dz\right)\,dy
			\end{multline*}
			which is nothing but an integral of the derivative of $\phi$, hence bounded on $\R_+$ since $\varphi$ is bounded. Therefore, we conclude on the integrability of $p_0(x)$ at infinity, completing the proof that $p_0$ defines a non-trivial invariant probability distribution. }
			\item {for $\nu<0$, we have noted that $g(x)/\gamma^2(x)$ is integrable at infinity under assumption~\ref{assumption:H4}, ensuring that the exponential term is upperbounded, hence integrability of $p_0$ relies of integrability of $\gamma^{-2}(x)$ at infinity, which is always the case under our assumptions.}
		\end{itemize}
		We hence proved that acceptable solutions related to $K=0$ exist for $\lambda>0$ and no solution existed when $\lambda<0$. Moreover, there is no possible choice of $K\neq 0$ that can overcome the exponential divergence at zero. Indeed, the only possible choice of $K$ would be the only constant preventing divergence at zero, i.e.:
			\[p(1)=2\frac{K}{\sigma^2} \int_11^0 \exp\left(-2\int_1^y \frac{2\lambda z}{\gamma(z)^2} + 2\frac{g(z)}{\gamma(z)^{2}}\,dz\right)dy.\]
			However, the integral term of the righthand side diverges, hence the above relationship cannot be satisfied and there is no acceptable solution for $\lambda<0$. Eventually let us remark that for $\lambda>0$, any choice of $K$ prevents integrability at infinity. 

		For $\alpha=1$, we again analyze the integrability properties of $p_0(x)$. Around $x=0$, $p_0$ behaves like $\frac{x^{\frac{2\lambda}{\sigma^2}}}{\gamma(x)^{2}}\exp\left(2\int_.^x \frac{g(y)}{\gamma(y)^{2}}\,dy\right)$. The term in the exponential behaves like $\frac{\lambda}{\sigma^2} y^{\kappa}$. The integrability at zero hence depends on the exponent $-2(1-\frac{\lambda}{\sigma^2})$: $p_0$ is integrable if the exponent is strictly larger than $-1$, i.e. $2\lambda/\sigma^2>1$. In the case $2\lambda/\sigma^2<1$, the non-integrability is due to a polynomial divergence, and hence might be compensated by a suitable choice of $K$. The only possible choice of $K$ compensating the divergence at zero would correspond to: 
		\[p(1)=2K \int_1^0 y^{-2\lambda/\sigma^2}\exp\left(-2\int_1^y \frac{g(z)}{\gamma(z)^2 }\,dz\right)dy,\]
		and here again the integral diverges at zero, and hence there is no such solution.

		At infinity, the exact same analysis as done in the case $\alpha>1$ applies, and hence we obtain that $p_0$ is integrable at $\infty$ under assumption~\ref{assumption:H4}.
	\end{proof}

	\subsection{Heuristic discussion}
	We hence proved that as a function of the value of the coefficient $\alpha$, the system can be in one of two substantially different regimes: in the case $\alpha<1$ the singular point is always stable, and reached in finite time, and for $\alpha>1$ the stability of the singular point is not affected by the presence of noise and only depends on the stability of the fixed point in the noiseless system. The case $\alpha=1$ constitutes the transition between these two regimes, and in that case the stability of the singular point depends both on the eigenvalue of the Jacobian matrix of the drift at the singular point and on the level of noise. 

	In order to understand heuristically these stabilization and destabilization phenomena, let us focus on the behavior of the system in the neighborhood of the singular point. In the case $\alpha>1$, the diffusion coefficient vanishes faster than the drift coefficient at the singular point, and the system locally behaves as if there were no noise in the system. In particular, the stability of the singular point remains unchanged from that of the noiseless system. On the contrary, when $\alpha<1$, the diffusion coefficient 
	vanishes slower than the drift. When the system approaches the singular point, it is hence mainly driven by random fluctuations related to the diffusion coefficient. When the system is brought away from the singular point, noise increases, and when it is brought towards the singular point noise decreases. This is at the origin of a ratchet-like phenomenon: consider to points $A$ and $B$ close of the singular point, $A$ being closer than $B$. The time it takes to go from $A$ to $B$ with power diffusion coefficient is way larger than the time needed to go from $B$ to $A$. This behavior of the noise close of zero explains the stability in probability, and acts as a weak drift by `stabilizing' the solutions when they approach the singular point. The intermediate case $\alpha=1$ is precisely the transition where these two effects compensate. When the noise coefficient is large enough compared to the eigenvalue of the Jacobian matrix at the singular point, the noise effects described for $\alpha<1$ dominate and stabilize the fixed point, and on the contrary the deterministic phenomena govern the dynamics similarly to the case $\alpha>1$. 

	\section{Applications}\label{sec:applications}
	The detailed characterization of the solutions of equation~\eqref{eq:generalg} has several implications in applied mathematics. We focus here on generic descriptions of the dynamics of one-dimensional SDEs in the flavor of bifurcations theory. We will specifically discuss the behavior of SDEs locally reducible (in a sense that we make precise) to the pitchfork or saddle-node bifurcation with power diffusion coefficients. 

	\subsection{Analysis of the stochastic pitchfork bifurcation with H\"older diffusion}\label{sec:Pitch}
	In this section we first study the dynamics of an SDE with drift given by the normal form of the pitchfork bifurcation with H\"older diffusion coefficients, before addressing the universality of these behaviors. This problem was widely addressed in the theory of random dynamical system by several authors~\cite{crauel:99,arnold:98} in the case of multiplicative noise ($\alpha=1$ in our notations). We extend these results to the cases $\alpha\neq 1$, and show that actually the multiplicative noise case is a singular transition case when varying $\alpha$. 
	
	\paragraph{Dynamics of the canonical stochastic pitchfork equation}
	We consider the supercritical stochastic pitchfork bifurcation with power diffusion. 
	\begin{equation}\label{eq:PitchAlpha}
		dx_t=(\lambda x_t-x_t^3)\,dt + \sigma \vert x_t\vert^{\alpha}dW_t
	\end{equation}
	for $x_t\in\R$, with initial condition $X_0$ at $t=0$. In this equation, $\lambda$ is a real parameter, $\sigma$ is a non-negative parameter and $\alpha\geq \frac 1 2 $. In the case $\sigma=0$, it is well known that the solution $x=0$ is a stable fixed point for any $\lambda<0$, and unstable for $\lambda >0$, and two additional stable equilibria $\pm \sqrt{\lambda}$ exist in the region $\lambda >0$ (see e.g.~\cite{guckenheimer-holmes:83}).

	For $\sigma\neq 0$, the model clearly satisfies assumptions~\ref{assumption:H1}-\ref{assumption:H5} with $\nu=-\mu=1>0$, $\kappa=\beta=2$, $d=\sigma$ and $\delta=\alpha$. Moreover, it satisfies the additional assumptions used in proposition~\ref{pro:QSD} since $g(x)=-x^3$ is differentiable with bounded derivative at zero and divergence at infinity upperbounded by $x^4$. The analysis of the model is therefore a direct application of section~\ref{sec:general}. 
	
	The problem was widely analyzed in the literature in the case $\alpha=1$ (see e.g.~\cite{arnold:98,crauel:99}). In this multiplicative noise case, it was proved that:
	\begin{itemize}
		\item for $\lambda<\sigma^2/2$, $\delta_0$ is almost surely exponentially stable and no additional stationary solution exist. 
		\item for $\lambda \geq \sigma^2/2$, $\delta_0$ is unstable in probability, and two stationary solutions stable in the first approximation exist. 
	\end{itemize}
	Therefore, in the multiplicative noise case, stability of zero depends on the level of noise. 
	
	Our analysis shows that this is a singular phenomenon that disappears upon variation of the H\"older exponent $\alpha$: 
		\begin{itemize}
			\item For $\alpha<1$, the fixed point $\delta_0$ is stable in probability. Moreover, any solution almost surely reaches zero in finite time. Quasi-stationary distributions exists for $\alpha\in (1/2,3/4)$.
			\item for $\alpha>1$ and: 
			\begin{itemize}
				\item $\lambda<0$, $\delta_0$ is stable in probability and no other stationary solution exist. 
				\item $\lambda >0$, $\delta_0$ unstable in probability and stationary solutions exist.
			\end{itemize}
		\end{itemize}

	Let us now further describe the dynamics of the non-trivial solutions as a function of $\alpha$.

	\begin{itemize}
		\item {\bf Case $\mathbf{\alpha=1}$:} It was observed in the literature (see e.g.~\cite[Chap. 9]{arnold:98} for a comprehensive account) that for any $\lambda>0$, the unstable deterministic fixed point becomes asymptotically exponentially stable when the noise parameter $\sigma$ is large enough.  For $\lambda>\frac{\sigma^2}{2}$, two symmetrical stationary distributions appear. For $\frac{\sigma^2}{2}<\lambda\leq \sigma^2$ the stationary distribution concentrates at zero and has a non increasing density diverging at zero. For $\lambda \geq \sigma^2$, the probability density function of the stationary distribution vanishes at zero and has a unique maximum reached for $x=\pm \sqrt{\lambda -\sigma^2}$ (see Figure~\ref{fig:PitchNoise}). There is hence a qualitative transition at $\lambda=\sigma^2$, or $P$-bifurcation. In comparison with the deterministic bifurcation, the loss of stability is delayed and noise tends to stabilize the saddle point.

		\item{\bf Case $\mathbf{\alpha>1}$:} the deterministic picture is qualitatively and quantitatively recovered: for $\lambda<0$, $\delta_0$ is stable in probability and is the unique stationary solution, and for $\lambda>0$, $\delta_0$ is unstable in probability, two additional stationary solutions appear, presumably stable as shown for $\alpha=2$. 
		The stationary probability density functions for $\lambda>0$ vanish at zero and reach a maximum at $X_m$ the solution of 
		\[\left( {\sigma }^{2}\alpha -\lambda \,X^{2-2\,\alpha }+X^{-2\,\alpha +4} \right)=0. \]
		The behavior of $x_m$ when $\lambda$ is close to zero follows one of the two regimes: 
		\begin{itemize}
			\item for $\alpha<2$, $X_m\sim \Big(\frac{\lambda}{\alpha\sigma^2}\Big)^{1/(2\alpha-2)}$. For instance for $\alpha=\frac{3}{2}$ we get  $X_m=-\frac{3}{4}\,{{\sigma}}^{2}+\frac 1 4\,\sqrt {9\,{{\sigma}}^{4}+16\,{\lambda}}
			$
			\item for $\alpha>2$, $X_m\sim \sqrt{\lambda}$ and is insensitive to noise parameters $\sigma$ and $\alpha$. For $\alpha=3$ we obtain the explicit form $X_m=\frac 1 {\sqrt{6}} \frac{\sqrt{-1+\sqrt{1+12\sigma^2\lambda}}}{\sigma}$
			\item for $\alpha=2$, this maximum can be computed explicitly, and is reached for $x=\sqrt{\frac{\lambda}{1+2\sigma^2}}$
		\end{itemize}
		Surprisingly, in the limit $\lambda\to 0$ there is a discontinuity in the behavior of $X_m$ as a function of $\alpha$ at $\alpha=2$. Nevertheless, when $\lambda>0$ this discontinuity is smoothed out. We conclude that in the case $\alpha>2$, the behavior of the system is close from the one of the deterministic pitchfork bifurcation and the stationary distribution escapes zero as a square root of the parameter. But when $1<\alpha<2$, this is no more the case: the scaling of distribution peak behaves as $\lambda^{1/(2\alpha-2)}$, hence much slower than $\sqrt{\lambda}$. As $\sigma$ is increased, the peak of the stationary distribution gets closer to zero. 
		\begin{figure}[!h]
			\centering
				\includegraphics[width=.7\textwidth]{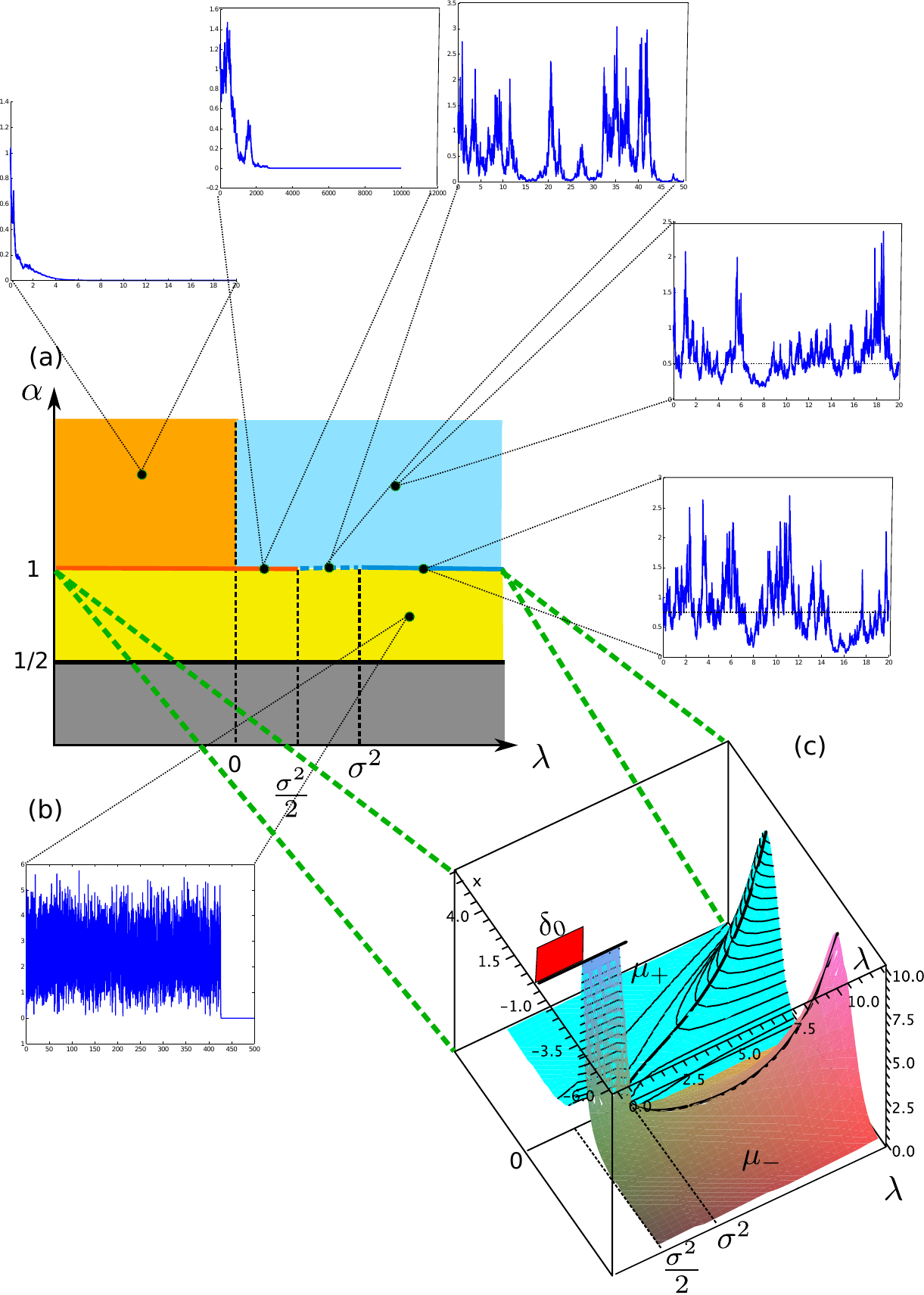}
			\caption{(a) Codimension 2 stochastic bifurcation diagram of the pitchfork bifurcation: orange: zero is stable, blue: zero is unstable, yellow: zero is stable and is reached in finite time. (b) Sample path trajectories in all 6 regimes for the stochastic pitchfork bifurcation. Parameters $(\lambda,\alpha,\sigma)$: (b1) $(10,0.6,1)$ (quasi-stationary distribution plotted in Fig.~\ref{fig:QSD}(a)) (b2): $(2,1,1.2)$, (b3): $(1,1.5,1.2)$, (b4): $(0.5,1,1.2)$, (b5): $(-1,1.5.1.2)$. In (b2) and (b3) the peak of the stationary distribution is depicted in dashed black. (c) Representation of stationary densities for the stochastic pitchfork system with linear noise, as a function of the parameter $\lambda$. Red square: $\delta_0$ is stable. Blue and colored surfaces display the probability distribution $p_{\lambda}(x)$ of the stable stationary distributions, and the black line represents the maximal value of the distribution.}
			\label{fig:PitchNoise}
		\end{figure}
		
		\item {\bf Case $\mathbf{\alpha<1}$:} the solution $\delta_0$ is always stable in probability. This result can appear relatively surprising at first sight. Indeed, in the deterministic case, $\lambda$ is the exponential rate of divergence from the solution $0$. However, adding a (possibly small) diffusion term proportional to $\vert x \vert^{\alpha}$ with $\alpha<1$ stabilizes $\delta_0$ in probability whatever the value of the noise intensity $\sigma\neq 0$. This observation, added to the fact that the solution is not exponentially asymptotically stable (though stable in probability) raises the question of how this convergence occurs. Figure \ref{fig:PitchNoise}(b1) presents a typical sample path of the process. Starting from a positive initial condition, we observe that the solution is evolving in the half-plane $x>0$ and does not show any absorption clue. However, it suddenly reaches zero where it is absorbed after that random transient phase. This perfectly illustrates a typical behavior of the solutions of the pitchfork equation for $\alpha<1$: the absorption time was shown to be almost surely finite. Moreover, trajectories that did not hit zero are distributed according to a quasi-stationary distributions as long as $\alpha<3/4$.
	\end{itemize}
	
	\paragraph{Universality properties}
	Similarly to the analysis of nonlinear dynamical systems, the analysis of the pitchfork equation reveals universal features of a wide class of systems. We formally show the universality property of the pitchfork bifurcation for general SDE on $\R$:
		\[dx_t=f(x_t,\lambda)\,dt+\gamma(x_t)dW_t\]
		with a smooth $f$ (at least three times continuously differentiable). We consider the dynamics of the system for parameter values $\lambda$ in a neighborhood of $\lambda_0\in\R$. Let us assume that:
		\begin{itemize}
			\item there exists $\nu>0$ and $\kappa>0$ such that for all $\lambda$ in a neighborhood of $\lambda_0$, $f(x,\lambda)\leq -\nu x^{1+\kappa}$,
			\item for any $\lambda$, $f$ is an odd function: $f(-x,\lambda)=-f(x,\lambda)$
			\item $\frac{\partial^3 f}{\partial x^3} (0,\lambda_0)\neq 0$
		\end{itemize}
	Defining $\xi_t=\beta (\lambda) x_t$ with $\beta(\lambda)=\sqrt{\frac 1 6 \left \vert \frac{\partial^3 f}{\partial x^3}(0,\lambda)\right \vert}$, we have:
	\[d\xi_t = \left(\derpart{f}{x}(0,\lambda)\xi_t+\varepsilon \xi_t^3 + \Psi(\xi_t,\lambda)\right)\,dt + \beta(\lambda) \gamma\left(\frac{\xi_t}{\beta(\lambda)}\right)\,dW_t.\]
	where $\Psi(x,\lambda)=O(x^5)$ and $\varepsilon=\textrm{sign} (\frac{\partial^3 f}{\partial x^3}(0,\lambda))$ which is well defined around $\lambda=\lambda_0$.

		Let us now assume that $\gamma(x)\sim \sigma \vert x \vert^{\alpha}$ around $x=0$ and $\gamma(x)>0$ for $x\neq 0$. We can hence apply the results developed in the previous section to characterize the dynamics of this general system. Pitchfork-like transitions occur at $\lambda=\lambda_0$ when:
		\begin{itemize}
			\item $\alpha>1$, $\derpart{f}{x}(0,\lambda_0)=0$ and $\frac{\partial^2f}{\partial x\partial \lambda}(0,\lambda_0)\neq 0$ or
			\item $\alpha=1$, $\derpart{f}{x}(0,\lambda_0)=\frac{\sigma^2}{2}$ and $\frac{\partial^2f}{\partial x\partial \lambda}(0,\lambda_0)\neq 0$. 
		\end{itemize}

	\subsection{The stochastic saddle-node bifurcation}\label{sec:SN}
	Another prominent universal codimension one bifurcation of equilibria in deterministic ODEs is the \emph{saddle-node (or fold) bifurcation}. It corresponds to the normal form $\dot{x}=-x^2+a$ (sometimes considered  $x^2-a$ which is equivalent to the normal form considered here  through a change of time). This dynamical system has no equilibrium when $a<0$ (solutions blows up in finite time) and two equilibria for $a>0$: $\sqrt{a}$ which is stable and $-\sqrt{a}$ which is unstable. The system has multiple singular points, and these depends on the bifurcation parameter. Choices of diffusion functions vanishing only at $\sqrt{a}$ or at $-\sqrt{a}$ enter the application domain of the general theory developed in section~\ref{sec:general}. We will consider a slightly different case in this section.
	\paragraph{Dynamics of a canonical stochastic saddle-node equation} Let us consider now a case, contrasting with the previous cases treated,  where the diffusion function vanishes at both singular points:
	\begin{equation}\label{eq:SaddleNode}
		dX_t=(-X_t^2+a) \,dt + \sigma \vert X_t^2-a \vert^{\alpha}\,dW_t
	\end{equation}
	For $a>0$, there are two stationary distributions corresponding to the equilibria $\sqrt{a}$ and $-\sqrt{a}$. Denoting $y^+_t = X_t-\sqrt{a}$ (vanishing at $X=\sqrt{a}$) and  $y^-_t = X_t+\sqrt{a}$, it is easy to show that these variable satisfy the equations:
	\[\begin{cases}
		dy^+_t=(-2\sqrt{a} y^+_t -(y_t^+)^2)\,dt + \sigma\vert-2\sqrt{a} y^+_t -(y_t^+)^2\vert^{\alpha}\,dW_t \\
		dy^-_t=(2\sqrt{a} y^-_t -(y_t^+)^2)\,dt + \sigma\vert 2\sqrt{a} y^+_t -(y_t^+)^2\vert^{\alpha}\,dW_t
	\end{cases}\]
	In both case, the drift and diffusion coefficients satisfy the assumptions~\ref{assumption:H1}, ~\ref{assumption:H2}, ~\ref{assumption:H3} and~\ref{assumption:H4} and obviously assumption~\ref{assumption:H5} is not valid. Assumptions~\ref{assumption:H1}, ~\ref{assumption:H2}  and~\ref{assumption:H4} are trivial since the function corresponding to $g(x)$ in the general cases is equal to $-x^2$, hence $\mu=-\nu=-1<0$ and $\beta=\kappa=1$. The diffusion  coefficients denoted $\gamma(x)$ in the general case satisfy for both the equation on $y^+$ and on $y^-$: $\gamma(x)\sim_{x\to 0^+} \sigma (2\sqrt{a})^{\alpha}x^{\alpha}$, ensuring that assumption~\ref{assumption:H3} is satisfied, and $\gamma(x)\sim_{x\to\infty} \sigma \vert y\vert^{2\alpha}$ ensuring that assumption~\ref{assumption:H4} is satisfied with $d=\sigma$ and $\delta=2\alpha$.

	In the deterministic case, for $a>0$ and an initial condition smaller than $-\sqrt{a}$ or for any initial condition when $a<0$, the solutions blow up in finite time to $-\infty$, and in usual applications additional, higher order confining terms prevent the blow up. Interestingly, in the present case, the presence of noise can prevent blow up as we now show:
	\begin{proposition}\label{pro:NoBlowUp}
		The solutions of the stochastic saddle-node equation~\eqref{eq:SaddleNode} are defined for all times or blow up in finite time under the conditions:
		\begin{itemize}
			\item For $\frac 1 2 \leq \alpha< 1$ we have:
			\begin{itemize}
				\item if $a<0$ the solutions almost surely blow up in finite time to $-\infty$
				\item if $a>0$, the solutions with initial condition smaller than $-\sqrt{a}$ almost surely exit the interval $(-\infty,-\sqrt{a})$ in finite time for any initial condition $\vert x \vert <\sqrt{a}$, the solution almost surely reaches one of the boundaries of the interval in finite time. The probability of reaching $-\sqrt{a}$ prior to reaching $\sqrt{a}$ is given by
				\[ \frac{p(\sqrt{a}^-) - p(x)}{p(\sqrt{a}^-) - p(-\sqrt{a}^+)}\] 
				and is plotted in Fig.~\ref{fig:SNBifs}(c).
			\end{itemize} 
			\item for $\alpha \geq 1$, the solutions almost surely never blow up, and for $a>0$, the exit time of $(-\sqrt{a},\sqrt{a})$ is almost surely infinite.
		\end{itemize}
	\end{proposition}

	\begin{proof}
		We demonstrate this proposition using Feller's test for explosion. 

		Let us start by considering $a<0$. The diffusion is then defined on $\R$ with no singular point, and Feller's scale function reads:
		\[p(x)=\int_{c}^{x} \exp\left(-2\int_c^y -(\xi^2+a)^{1-2\alpha}\,d\xi \right)dy.\]
		At infinity, the integrand inside the exponential is equivalent at infinity to $\frac{2}{2(2\alpha-1)-1} y^{-2(2\alpha-1)+1}$. This quantity has a finite limit when $\alpha\geq 1 $, implying that $p(x)\to \pm \infty$ when $x\to \pm \infty$. Feller's test implies that the first exit time of $(-\infty,\infty)$ is infinite with probability one. 

		For $\alpha<1$, $p(x)$ tends to $\infty$ at $\infty$, and has a finite limit at $-\infty$. This property ensures that the probability of the process diverge towards $-\infty$ is equal to one. It remains to prove that the explosion occurs in finite time. To this purpose, we need to show that the limit of Feller's scale function $v(x)$ is finite at $-\infty$. This is the case under the assumption $\alpha<1$. Indeed, $v(x)$ is defined as:
		\[v(x)=\int_c^x \int_c^y \exp\left(2\int_z^y (u^2-a)^{1-2\alpha}\right)\]
		The integral in the argument of the exponential function is equivalent at $-\infty$ to $\frac 2 {\Phi} (y^{\Phi}-z^{\Phi})$ which tends to $-\infty$ as a power function, hence the exponential has moments of any order, and in particular is twice integrable at $-\infty$.

		For $a>0$, we consider the intervals $(-\infty,-\sqrt{a})$ and $(\sqrt{a},\infty)$. The above analysis proved that for $\alpha\geq 1$, $p(x)$ tends to $\infty$ when $x\to\infty$ and towards a finite value at $x\to-\infty$. At $\pm \sqrt{a}$, the argument of the exponential term is equivalent to $\int x^{-2\alpha+1}$ which diverges for $\alpha\geq 1$, implying that the exit time of $(-\infty, -\sqrt{a})$, $(-\sqrt{a},\sqrt{a})$ and $(\sqrt{a},\infty)$ are all almost surely infinite and that there is no blow up in finite time. 
		For $\alpha<1$, both $p(-\sqrt{a}^-)$ and $p(\sqrt{a}^+)$ are finite. For $c\in(-\sqrt{a},\sqrt{a})$, we have:
		\[\begin{cases}
			G(x)=\int_{c}^x -(\xi+\sqrt{a})^{1-2\alpha}(-\xi+\sqrt{a})^{1-2\alpha}\,d\xi\\
			v(x)=\displaystyle{2\int_c^x\int_y^x \frac{\exp(-2(G(y)-G(z)))}{(-x^2+a)^{2\alpha}}\,dz \,dy}
		\end{cases}\]
		and using a similar argument as used in the proof of proposition~\ref{pro:firsthittingtimezero}, we show that $\lim_{x\to\sqrt{a}^-} \,v(x)<\infty$, and by symmetry we hence have $\lim_{x\to -\sqrt{a}^+} \,v(x) <\infty$. Feller's test for explosion hence ensures that the first exit time of the interval $(-\sqrt{a},\sqrt{a})$ is almost surely finite. We hence converge in finite time towards $\sqrt{a}$ or $-\sqrt{a}$. Moreover, the probability for a trajectory starting with an initial condition $\vert x\vert <\sqrt{a}$ to reach $\sqrt{a}$ (resp. $-\sqrt{a}$) prior to reaching $-\sqrt{a}$ (resp. $\sqrt{a}$), denoted  $p_+(x)$ (resp. $p_-(x)$) have the expression: 
		\[ p_-(x) = 1-p_+(x) = \frac{p(\sqrt{a}^-) - p(x)}{p(\sqrt{a}^-) - p(-\sqrt{a}^+)}\]
		This quantity is plotted in Figure~\ref{fig:SNBifs}(c).
	\end{proof}

	Application of the general results of section~\ref{sec:general} yield the following:

	\begin{proposition}\label{pro:SN}
		For $a>0$, the saddle-node equation~\eqref{eq:SaddleNode} has two Dirac delta distributed functions $\delta_{\sqrt{a}}$ and $\delta_{-\sqrt{a}}$, which enjoy the stability properties: 
		\begin{itemize}
			\item for $\alpha<1$, both $\delta_{\sqrt{a}}$ and $\delta_{-\sqrt{a}}$ are stable in probability, and the first hitting time $\tau$ of one of the equilibria is almost surely finite. 
			\item for $\alpha=1$, the stationary distribution $\delta_{\sqrt{a}}$ is always stable in probability. The stationary distribution $\delta_{-\sqrt{a}}$ is stable in probability for $a>\sigma^{-4}$ and unstable in probability otherwise. 
			\item for $\alpha>1$, the stationary distribution $\delta_{\sqrt{a}}$ is stable in probability and $\delta_{-\sqrt{a}}$ is unstable in probability.
		\end{itemize}
	\end{proposition}

	\begin{proof}
		This is a direct application of the results of section~\ref{sec:general} applied to $y^+$ to characterize the properties of $\sqrt{a}$ and to $y^-$ for $-\sqrt{a}$. For instance for the fixed point $-\sqrt{a}$ and $\alpha=1$, the fixed point is stable if and only if $2\sqrt{a}<2 \sigma^2 \alpha$ i.e. $\sigma^2\sqrt{a}>1$. The finite absorption time property is a consequence of proposition~\ref{pro:NoBlowUp}.

	\end{proof}

	For $\alpha<1$, we know that the trajectories either reach one of the equilibria or blow up in finite time, which describe perfectly the permanent regime of the saddle-node bifurcation. We are now interested in stationary distributions of the saddle-node equation distinct from $\delta_{\sqrt{a}}$ and $\delta_{-\sqrt{a}}$. Since the system presents two singular points, the results of proposition~\ref{pro:Station}, based on integrability properties on $\R^+$ or $\R^-$ of the solutions of the Kolmogorov equation, do not directly apply. However, a similar approach allows to demonstrate the following:

	\begin{proposition}\label{pro:SNStationary}
		Stationary distributions of the saddle-node equation~\eqref{eq:SaddleNode} distinct from $\delta_{\sqrt{a}}$ and $\delta_{-\sqrt{a}}$ (when $a>0$) enjoy the following classification:
		\begin{itemize}
			\item For $\alpha=1$ and $0<a<\sigma^{-4}$, the system has a stationary distribution, stable in the first approximation, charging only the interval $(-\infty,-\sqrt{a})$ that undergoes a P-bifurcation at $\sigma^2\sqrt{a}=\frac 1 2 $. On $(-\sqrt{a},\infty)$, any solution converges almost surely exponentially towards $\sqrt{a}$ (see Figure~\ref{fig:SNBifs}). 
			\item For $\alpha=1$ and $a<0$, there exists a unique distribution charging the whole real line.
			\item For $\alpha>1$, the system has no additional stationary solution. 
		\end{itemize}
	\end{proposition}
	\begin{proof}
	 In the case $\alpha=1$, the Kolmogorov equation has the solution:
		\[p_0(x)\propto \left\vert x+\sqrt {a} \right\vert ^{-2+{\frac {1}{{\sigma}^{2}\sqrt {a}}}}
		\left\vert x-\sqrt {a} \right\vert ^{-2-{\frac {1}{{\sigma}^{2}\sqrt {a}}}}
		\]
		For $a>0$, this function is integrable at $\sqrt{a}$ if and only if $-2-\frac{1}{\sigma^2\sqrt{a}} >-1$ which is not possible. At $-\sqrt{a}$, this function is integrable when ${-2+{\frac {1}{{\sigma}^{2}\sqrt {a}}}}>1$ i.e. $\sigma^2\sqrt{a}<1$ (i.e. when $-\sqrt{a}$ is unstable). Eventually, this function is integrable at infinity whatever the parameters, since it behaves like $x^{-4}$. Hence there exists a unique stationary distribution charging $(-\infty,-\sqrt{a})$. In both the interval $(-\sqrt{a},\sqrt{a})$ and $(\sqrt{a},\infty)$ the only stationary distribution is $\delta_{\sqrt{a}}$ which is stable. 
		The additional stationary distribution on $(-\infty,-\sqrt{a})$ reads:
		\[p_0(x) = Z \left(- x-\sqrt {a} \right) ^{-2+{\frac {1}{{\sigma}^{2}\sqrt {a}}}}
		\left( -x+\sqrt {a} \right) ^{-2-{\frac {1}{{\sigma}^{2}\sqrt {a}}}}=:Z\,q_0(x)
		\]
		Denoting $\Phi={\frac {1}{{\sigma}^{2}\sqrt {a}}}$, we give the following closed form primitives related:
		\[
		\begin{cases}
			\displaystyle{\int q_0(x)\,dx} &= \displaystyle{-(-\sqrt{a}-x)^{-1+\Phi}(\sqrt{a}-x)^{-1-\Phi}\frac{a(2\Phi^2-1)-2\sqrt{a}\Phi x + x^2}{4a^{3/2}\Phi (\Phi^2-1)}}\\
			\displaystyle{\int x\,q_0(x)\,dx} &= \displaystyle{(-\sqrt{a}-x)^{-1+\Phi}(\sqrt{a}-x)^{-1-\Phi}\frac{-2\sqrt{a}\Phi x + a + x^2}{4a(\Phi^2-1)}}\\
			\displaystyle{\int x^2\,q_0(x)\,dx} &= \displaystyle{-(-\sqrt{a}-x)^{-1+\Phi}(\sqrt{a}-x)^{-1-\Phi}\frac{-2\sqrt{a}\Phi x + a + (2\Phi^2-1)x^2}{4\sqrt{a}\Phi(\Phi^2-1)}}\\
		\end{cases}
		\]
		Using these formulae, it is straightforward to obtain the normalization constant of $p_0$ on $(-\infty,-\sqrt{a})$: 
		\[Z=4a^{3/2} \Phi (\Phi^2-1)\]
		The expectation $m$ is equal to:
		\[m=Z\int_{-\infty}^{-\sqrt{a}}xq_0(x)\,dx = -\Phi \sqrt{a}=-\frac 1 {\sigma^2},\]
		and the second moment:
		\[s=Z\int_{-\infty}^{-\sqrt{a}}x^2q_0(x)\,dx ={(2\Phi^2-1) a }\]
		Eventually, let us analyze the shape of the obtained stationary distribution. The differential of the distribution reads:
		\[p_0'(x) = p_0(x)\left(\frac{-2-\Phi}{x-\sqrt{a}} + \frac{-2+\Phi}{x+\sqrt{a}}\right).\] 
		This quantity is always negative for $\Phi<2$, the density is decreasing and moreover in that case it is clear that the distribution diverges at $-\sqrt{a}$. For $\Phi>2$, the distribution vanishes at zero, takes its maximum at $x^*=-\frac{\Phi\sqrt{a}}{2}=-\frac 1 {2\sigma^2}$ and goes back to zero. 

		Let us now address the stability in the first approximation for this distribution. The linearized equation reads:
		\[dv_t=-2X_tv_t\,dt+2\sigma X_tv_tdW_t\]Integrating the linearized equation as previously done, we obtain that the Lyapunov exponent of the linearized equation reads:
		\[l=-2\Exp[X_t] - 2\sigma^2 \Exp[X_t^2]=\frac{2(-1+\sigma^4\,a)}{\sigma^2}\]
		and since we consider the case $\sigma^2\sqrt{a}<1$, the Lyapunov exponent is strictly negative, implying almost sure exponential stability of zero for the linearized equation, hence stochastic linear stability of $p_0(x)$. 

		For $a<0$, the solutions of the Kolmogorov equation are given by:
		\[\displaystyle{p_0(x)=\frac{K }{(x^2-a)^2}}\exp\left({-2\frac{\arctan(x/\sqrt{-a})}{\sigma^2\sqrt{-a}}}\right)\]
		and the definition interval of possible stationary distributions is $\R$. In order for $p_0$ to define a probability distribution, we need this function to be integrable at $\pm \infty$, which is always the case. The obtained stationary distribution reaches its maximum again for $x=-\frac{1}{2\sigma^2}$.

		For $\alpha>1$, the solution of the Kolmogorov equation reads:
		\[p(x)=\frac 1 {(a-x^2)^{2\alpha}}\exp\left(\frac 2 {\sigma^2} \int_{\cdot}^x (-y^2+a)^{1-2\alpha}\,dy\right)\]
		At $-\sqrt{a}$, the exponential term diverges since it behaves as $(x+\sqrt{a})^{1-2\alpha}$, hence the solution is never integrable at zero. It is easy at this point to show that there is no integrable solution to the Kolmogorov equation.
		
	\end{proof}
	\begin{remark}
		For $\alpha <1$ one could be interested in the existence of quasi-stationary solutions. An additional difficulty arises from the fact that the solutions can be attracted in finite time towards different solutions: $\pm\sqrt{a}$ whatever $\alpha<1$, and $-\infty$ for $\alpha <\frac 3 4$. Moreover, in the saddle-node bifurcation case as studied here, the diffusion coefficient is exactly equal to a power function. Therefore, for $\alpha<1$, blow up in finite time as well as absorption at $\pm \sqrt{a}$ can occur are not a consequence of the results of~\cite{cattiaux-collet-etal:09}, are not in the scope of the present manuscript and require new mathematical developments that will be addressed elsewhere. 
	\end{remark}

	\begin{figure}[!h]
		\centering
			\includegraphics[width=.7\textwidth]{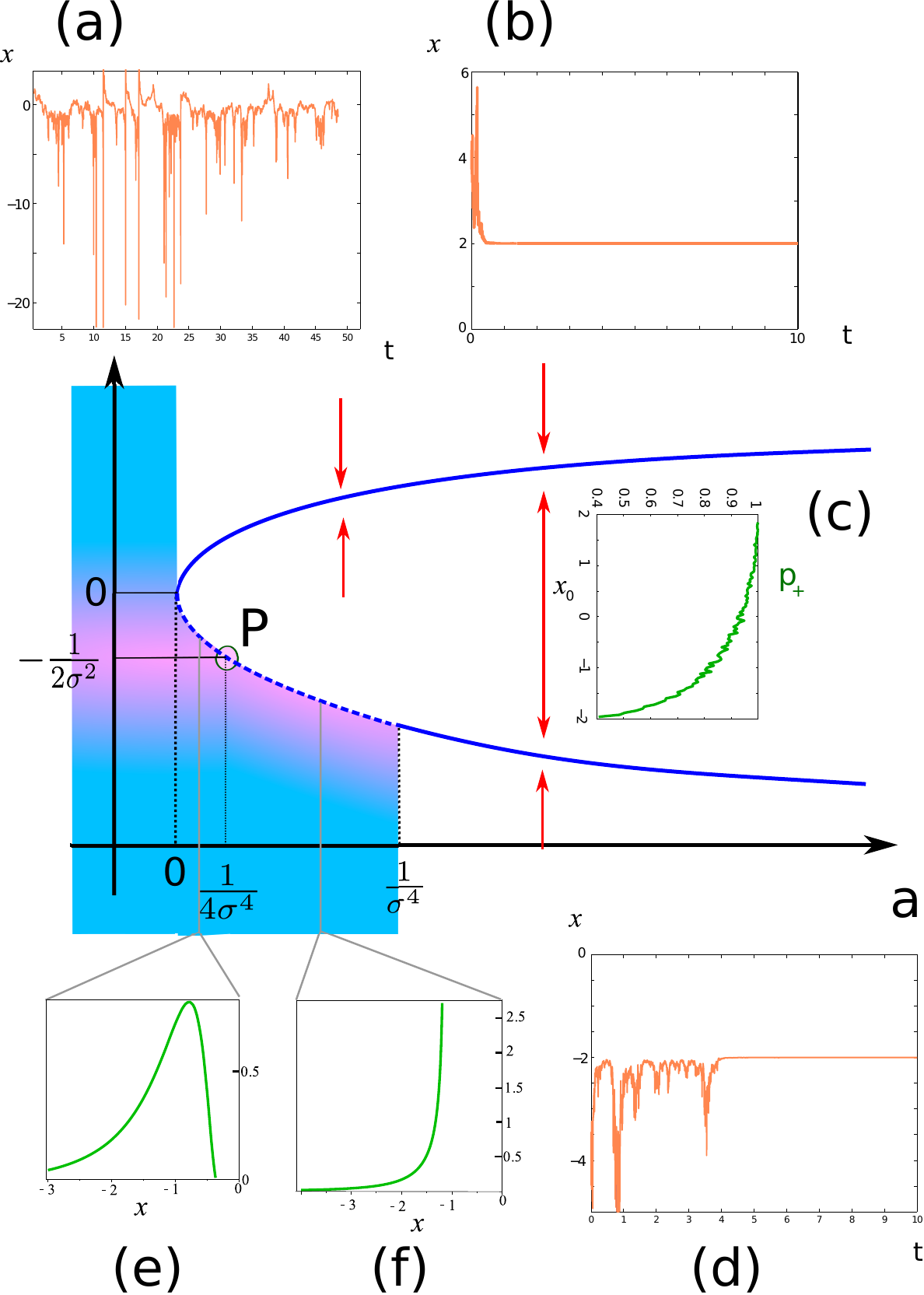}
		\caption{Behaviors of the stochastic saddle-node system with $\alpha=1$. The bifurcation diagram as a function of $a$ is at the center of the figure, and displays the singular points in blue lines, their stability marked by the type of curve: plain: stable, dashed: unstable. Stationary distributions are encoded in color in the interval $a\leq \sigma^{-4}$. Two typical stationary distributions are plotted in Figs. (e) and (f): one for $a\leq 1/4\sigma^4$ where the pdf has a peak at $x=-\sigma^{-2}/2$, and for $a> \sigma^{-4}/4$ where the distribution diverges at zero. The P-bifurcation between these behaviors is noted P. (a) corresponds to a typical trajectory for $a=0.4<\sigma^{-4}/4$ with $\sigma=0.8$. (b),(c) and (d) correspond to $a>\sigma^{-4}$ and different initial conditions: (c): frequency of the trajectories converging towards $+\sqrt{a}$ as a function of the initial condition in $[-\sqrt{a},\sqrt{a}]$, (b) and (d) are typical trajectories for initial condition in $(\sqrt(a),\infty)$ or $(-\infty,-\sqrt{a})$ respectively ($a=4$ and $\sigma=0.8$). }
		\label{fig:SNBifs}

	\end{figure}

	The above result characterize the dynamics relatively exhaustively for $\alpha=1$, except for $a>\sigma^{-4}$ and $x_0\in(-\sqrt{a},\sqrt{a})$. In that case, both singular points are stable, and can be reached by the solution. Therefore for such initial conditions, there is a competition between the possible attractors, and the relative stability of one singular point compared to the other will determine the probability, given an initial condition $x_0$, to converge towards this equilibrium. This quantity is plotted in Figure~\ref{fig:SNBifs}(c). We observe, as one can expect, that the attractivity of $\sqrt{a}$ is stronger than that of $-\sqrt{a}$: the probability of converging towards $\sqrt{a}$ is never less than $0.4$ even for initial conditions very close from $-\sqrt{a}$, and goes relatively fast towards 1 as the initial condition is chosen closer from $\sqrt{a}$. 

	\paragraph{Universality properties} Let us now discuss the universality of this bifurcation. We consider a general one-dimensional diffusion equation:
	\[dx_t=f(x_t,\lambda)\,dt + g(x_t)dW_t\]
	with $f$ a smooth function such that $f(0,0)=\derpart{f}{x}(0,0)=0$ and $\derpart{f}{\lambda}(0,0)\neq 0$. Following the reduction to a normal form of Kuznetsov~\cite{kuznetsov:98}, we use a Taylor expansion of $f$:
	\[f(x,\lambda)=f_0(\lambda)+f_1(\lambda)\,x+f_2(\lambda)\,x^2+h(x)\]
	where $h(x)/x^3$ is bounded at $x=0$. We perform a shift of coordinates by defining a new variable $\xi=x+\delta$ and obtain:
	\begin{align*}
		d\xi_t&=\Big([f_0(\lambda)-f_1(\lambda)\delta]+f_2(\lambda)\delta^2+O(\delta^3)] \\
		&\quad + [f_1(\lambda)-2f_2(\lambda)\delta + O(\delta^2)]\xi\\
		&\quad +[f_2(\lambda)+O(\delta)]\xi^2 + O(\xi^3)\Big)\,dt + g(\xi-\delta)dW_t
	\end{align*}
	Using the inverse function theorem under the assumption that $\derpart{f}{\lambda}(0,0)\neq 0$, we deduce that there exists a smooth function $\delta(\lambda)= \frac{f_1'(0)}{2f_2(0)}\lambda+O(\lambda^2)$ such that the term $[f_1(\lambda)-2f_2(\lambda)\delta + O(\delta^2)]$ vanishes for all $\lambda$ in a neighborhood of $0$. This $\delta(\lambda)$ identified, we can apply the inverse function theorem to find locally an inverse to the constant term of the Taylor development seen as a function of $\lambda$. This inverse function, $\lambda(\mu)$, yields the system to the SDE:
	\[d\xi_t=(\mu + b(\mu)\xi^2_t + O(\xi^3))\,dt + g(\xi-\delta(\lambda(\mu)))dW_t \]
	with $b(0)=f_2(0)\neq 0$. The final reduction involves a scaling of the variable: $z_t=\vert b(\mu)\vert \xi_t$, yielding the SDE:
	\[dz_t=\Big(\beta + \varepsilon z_t^2+O(z_t^3)\Big)\,dt + g\Big(\frac{z_t}{\vert b(\mu)\vert}-\delta(\lambda(\mu))\Big)dW_t\]
	where $\varepsilon$ has the sign of $S^{(3)}(0)$. The stability and bifurcations of the whole system is hence reduced to the local analysis of the behavior of $g$ close of the equilibria. 
	
	\appendix
	\section{The stochastic Hopf bifurcation with complex noise.}
	{The stochastic Hopf bifurcation with multiplicative noise ($\alpha=1$) was investigated with RDS theory in~\cite{baxendale1994stochastic}. The method uses smoothness of the diffusion coefficient to characterize the changes in stability in a general case, and applies to characterize invariance measures in one specific case of the Hopf bifurcation. Here, we use our general results in a more general Hopf bifurcation model, and characterize emerging oscillations, in cases with arbitrary $\alpha \geq 1/2$.} The Hopf bifurcation for smooth dynamical system is topologically equivalent to the two dimensional normal form: 
	\[\der{x}{t} = \beta\,x - y + \varepsilon x(x^2+y^2) \qquad ; \qquad \der{y}{t} = \beta\,y +x + \varepsilon y(x^2+y^2) \]
	where $\beta$ is the bifurcation parameter and $\varepsilon=\pm 1$ governs the type of bifurcation. It can be conveniently written in complex form for $z_t=x_t+ \mathbf{i}\,y_t$ with $\mathbf{i}^2=-1$:
	\[\der{z}{t} = (\beta + \mathbf{i} + \varepsilon \vert z(t) \vert^2)\,z(t).\]
	The trivial solution $(x=0,y=0)$ is always solution of the equations. It is exponentially stable for any $\beta<0$ and exponentially unstable for any $\beta>0$. If $\varepsilon =-1$, the system presents stable oscillations for $\beta>0$ and the bifurcation is termed supercritical, and supercritical otherwise. 

	We study in this section the stochastic Hopf bifurcation with multiplicative noise. For simplicity, we will consider that the stochastic perturbations are driven by a single Brownian motion $W_t$. Let $\sigma^*=\sigma+\mathbf{i}\mu$ be a noise parameter, the stochastic Hopf normal form equation reads (in complex form):
	\[dZ_t = (\beta+i+\varepsilon \vert Z_t\vert ^2)Z_t \, dt + \sigma^* \vert Z_t\vert^{\alpha-1} Z_t \, dW_t,\] 
	which corresponds for $Z_t=X_t+\mathbf{i}Y_t$ to the bidimensional system:
	\begin{equation}\label{eq:Hopf}
		\begin{cases}
			dX_t &= (\beta\,X_t - Y_t + \varepsilon X_t\,(X_t^2+Y_t^2))\,dt + (\sigma\,X_t - \mu\,Y_t) \vert Z_t\vert^{\alpha-1}\,dW_t\\
			dY_t &= (\beta\,Y_t + X_t + \varepsilon Y_t\,(X_t^2+Y_t^2))\,dt + (\mu\,X_t +\sigma\, Y_t) \vert Z_t\vert^{\alpha-1}\,dW_t\\
		\end{cases}
	\end{equation}
	where $\beta$ and $\varepsilon$ correspond to the parameters of the Hopf bifurcation and $\mu$ and $\sigma$ to the amplitude of the stochastic perturbation. 

	The deterministic process $Z_t=0$ corresponding to $X_t=0,Y_t=0$ for all $t\geq 0$ is solution of \eqref{eq:Hopf} and is univocally defined by the fact that the modulus of $Z_t$ is null. We denote by $R_t=\vert Z_t \vert = \sqrt{X_t^2+Y_t^2}$ this variable, and by $\theta_t$ the argument of $Z_t$. 

	\begin{lemma}\label{lem:ItoHopf}
		The modulus of the variable $R_t=\vert Z_t \vert$ and the argument $\theta_t$ satisfy the equations:
		\[\begin{cases}
			dR_t&= \left( \beta R_t+\frac{\mu^2}{2}R_t^{2\alpha-1} + \varepsilon R_t^3 \right)\,dt + \sigma\,R_t^{\alpha}\,dW_t\\
			d\theta_t&= (-\sigma\mu R_t^{2\alpha-2} +1) \,dt+ \mu R_t^{\alpha-1}dW_t
		\end{cases}\]	
	\end{lemma}

	\begin{proof}
		Let $(X_t,Y_t)$ be a solution of the Hopf equations \eqref{eq:Hopf}. We apply It\^o formula to the variable $R_t=\sqrt{X_t^2+Y_t^2}$:
		\begin{align*}
			dR_t &= \frac{1}{R_t} (X_t dX_t + Y_t dY_t) + \frac 1 {2\,R_t^3} \left(Y_t^2 d\langle X \rangle_t + X_t^2 d\langle Y \rangle_t +2\,X_t\,Y_t d\langle X,Y \rangle_t \right)\\
			&= \Big\{ \frac{1}{R_t} \bigg(\beta\,X_t^2 - X_t\,Y_t + \varepsilon X_t^2\,(X_t^2+Y_t^2) + \beta\,Y_t^2 + X_t\,Y_t + \varepsilon Y_t^2\,(X_t^2+Y_t^2)\bigg) \\
			& \quad + \frac{R_t^{2\alpha-2}}{2\,R_t^3} \bigg( Y_t^2(\sigma\,X_t-\mu\,Y_t)^2 + X_t^2(\sigma\,Y_t+\mu\,X_t)^2 -2 X_t\,Y_t\,(\sigma\,X_t-\mu\,Y_t)\,(\sigma\,Y_t+\mu\,X_t)\bigg)\Big\}\,dt \\
			&\quad + \frac{R_t^{\alpha-1}}{R_t} \Big\{ X_t \, (\sigma\,X_t-\mu\,Y_t)+ Y_t \, (\sigma\,Y_t+\mu\,X_t)\Big\}\,dW_t\\
			&=\Big\{\frac{1}{R_t}(\beta R_t^2+\varepsilon R_t^4) + \frac{1}{2R_t^3}\mu^2 R_t^{4+2\alpha-2} \Big\}\,dt + \frac{1}{R_t}\sigma\,R_t^{2+\alpha-1}\,dW_t\\
			&= \left( \beta R_t + \frac{\mu^2}{2} R_t^{2\alpha-1}+ \varepsilon R_t^3 \right)\,dt + \sigma\,R_t^{\alpha}\,dW_t
		\end{align*}
		The argument $\theta_t$ is given by $\theta_t=\arctan(Y_t/X_t)$. Applying It\^o's formula again yields:
		\begin{align*}
			d\theta_t &= (1-\sigma\mu R_t^{2(\alpha-1)}) \,dt+ \mu R_t^{\alpha-1}dW_t
		\end{align*}
		which ends the proof of the lemma.
	\end{proof}

	It is important to note that the equation on the modulus is uncoupled of the phase equation on $\theta_t$. The modulus of $(X_t,Y_t)$ is therefore solution of a stochastic differential equation of type bifurcation~\eqref{eq:generalg}. Moreover, when $\alpha=1$, the equations take the simpler form:
	\[\begin{cases}
		dR_t&= \left( (\beta +\frac{\mu^2}{2}) R_t + \varepsilon R_t^3 \right)\,dt + \sigma\,R_t \,dW_t\\
		d\theta_t&= (1-\sigma\mu) \,dt+ \mu dW_t
	\end{cases}\]
	and hence the variable $R_t$ is solution of a stochastic pitchfork bifurcation. From the results of section~\ref{sec:Pitch}, it is not hard to establish the following:
	\begin{theorem}\label{theo:Hopf}
		The null solution of the supercritical Hopf equations is almost surely exponentially stable if $\beta<\frac{\sigma^2-\mu^2}{2}$ and asymptotically stochastically unstable if $\beta>\frac{\sigma^2-\mu^2}{2}$. In that case, there exists a new stochastically stable stationary solution with distribution:
		\[p_R(x)= \frac {2{\sigma^{1-\frac{2\lambda}{\sigma^2}}}} {\Gamma\left(-\frac 1 2 +\frac{\lambda}{\sigma^2}\right)} x^{-2(1-\frac{\lambda}{\sigma^2})}e^{-x^2/\sigma^2}\mathbbm{1}_{x\geq 0}.\]
		The null solution of the subcritical Hopf equations is almost surely exponentially unstable if $\beta >\frac{\sigma^2-\mu^2}{2}$. It is asymptotically stochastically unstable if $\beta > -\frac{\sigma^2+\mu^2}{2}$ and stochastically stable if $\beta < -\frac{\sigma^2+\mu^2}{2}$.
	\end{theorem}

	We observe in the purely imaginary noise case ($\sigma=0$), $R_t$ satisfies the deterministic pitchfork bifurcation equation. For $\beta>\frac{\mu^2}{2}$, $R_t=\sqrt{\beta -\frac{\mu^2}{2}}$ is the only attractive solution: the solution of the stochastic Hopf bifurcation asymptotically lives on the circle. Its rotates on this circle with a stochastic phase given by a Brownian motion multiplied by $\mu$ and a deterministic pulsation equal to $1$ (see Fig.~\ref{fig:Hopf}). When $\mu=0$, the phase equation is deterministic: the system rotates with pulsation $1$. When $\beta \in (\frac{2\sigma^2-\mu^2}{2},\frac{\sigma^2-\mu^2}{2})$, the stationary modulus of the solution has a distribution peaked at zero and no actual oscillation can be observed. When $\beta>\frac{2\sigma^2-\mu^2}{2}$, the stationary modulus of the solution has a maximum at $R^*=\sqrt{\beta-\frac{2\sigma^2-\mu^2}{2}}$ and the distribution is null at zero: solutions hence correspond to perfectly periodic oscillations with random amplitude. 

	\begin{figure}[!h]
		\centering
		\includegraphics[width=0.8\textwidth]{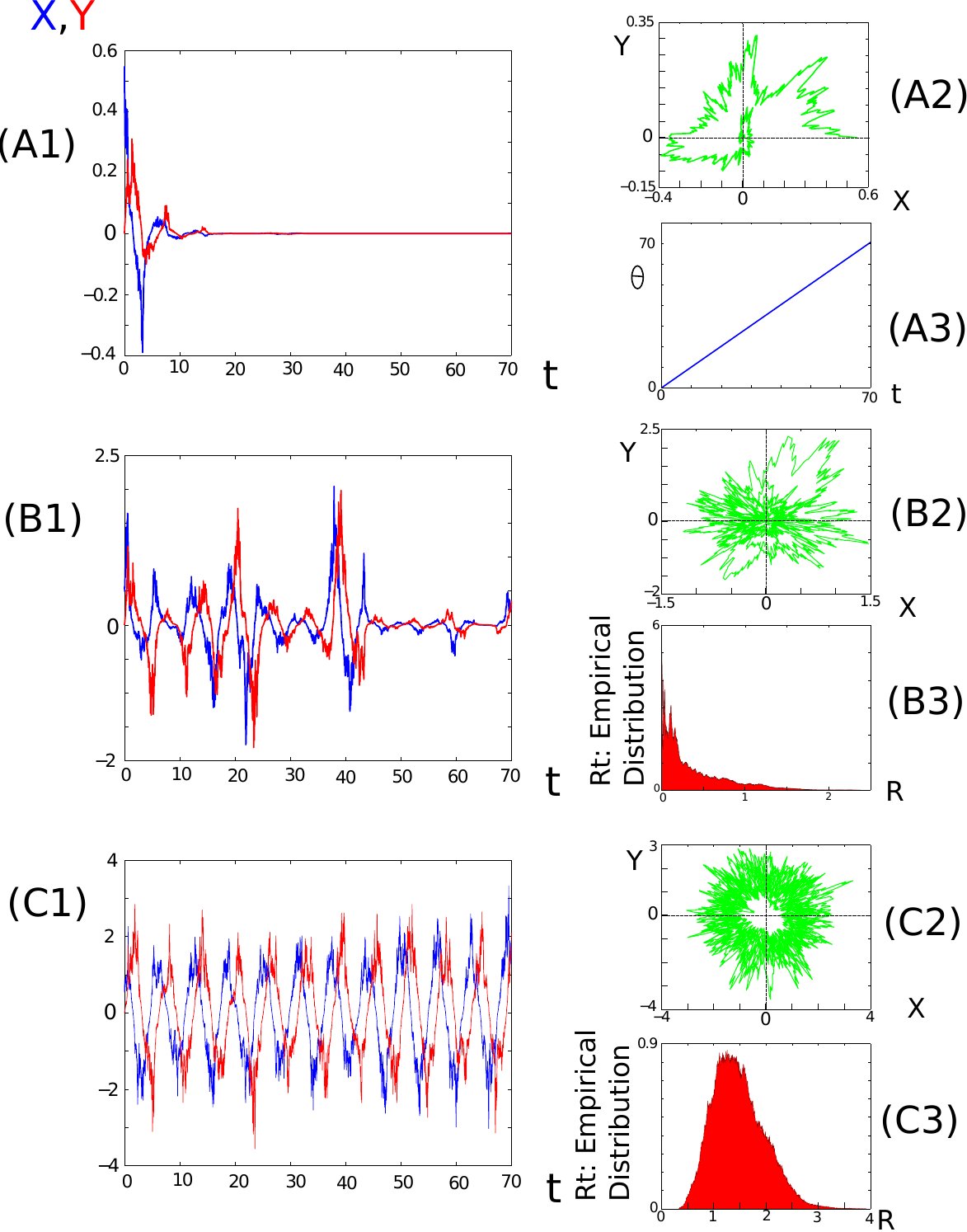}
		\caption{Simulations of sample path trajectories of $(X_t,Y_t)$ for the stochastic Hopf bifurcation with $\alpha=1$, $\mu=0$ and $\sigma=1$ for 
		 different values of $\beta$ (deterministic phase case). (A):$\beta=0.3<\frac{\sigma^2}{2}$: zero is a.s. exponentially stable; (A1): trajectories 
		(blue: $X_t$, red: $Y_t$), (A2): phase plane, (A3): phase $\theta_t$ is deterministic equal to $t$: the solutions rotate with pulsation one. This is 
		true of all three cases (A), (B) and (C). (B): $\beta=0.7\in[\frac{\sigma^2}{2},\sigma^2]$: Zero is unstable and the stationary modulus distribution 
		is peaked at zero; (B1): trajectories, (B2): phase plane, (B3): distribution of the modulus $R_t$. (C): $\beta=4>\sigma^2$: zero is unstable and the 
		distribution of the modulus has a peak around $R^*>0$ (C3): the trajectories display perfectly periodic oscillations (C1, C2) with random modulus, a highly regular behavior for a stochastic system.
		}
		\label{fig:Hopf}
	\end{figure}

	For $\sigma\neq 0$, the modulus $R_t$ is a.s. exponentially attracted towards $0$ when $\beta<\frac{\sigma^2-\mu^2}{2}$. Therefore, as the imaginary part of the noise, $\sigma$, is increased, the fixed point gains stability. For $\beta\in (\frac{\sigma^2-\mu^2}{2},\frac{2\sigma^2-\mu^2}{2})$, the fixed point zero is unstable and solutions do not converge towards the fixed point zero. The distribution of the moduli $R_t$ converge towards the stationary distribution $p_R(x)$ heavily charging zero, and solutions stay close from the fixed point with random stochastic variations. As soon as $\beta>\frac{2\sigma^2-\mu^2}{2}$, the solutions have radii centered at $R^*:=\sqrt{\beta-\frac{2\sigma^2-\mu^2}{2}}$ and present stochastic oscillations with a deterministic pulsation $(-1+\sigma\mu)$ and random phase given by a Brownian motion scaled by $\mu$. Interestingly, when $\sigma\mu$ crosses $1$, the direction of the oscillation changes: for $\sigma\mu<1$, the oscillation is counter clockwise, and for $\sigma\mu>0$, the oscillation is clockwise. This bifurcation is illustrated in Figure~\ref{fig:changeOscillation}. In the case $\sigma\mu=1$, the pulsation is null and the phase is given by $\mu W_t$: it has a probability $1/2$ to rotate clockwise and $1/2$ to rotate counterclockwise, and changes the sense of rotation. The rotation number is equal to zero.

	\begin{figure}[!h]
		\centering
			\subfigure[$\sigma=0.5$]{\includegraphics[width=.3\textwidth]{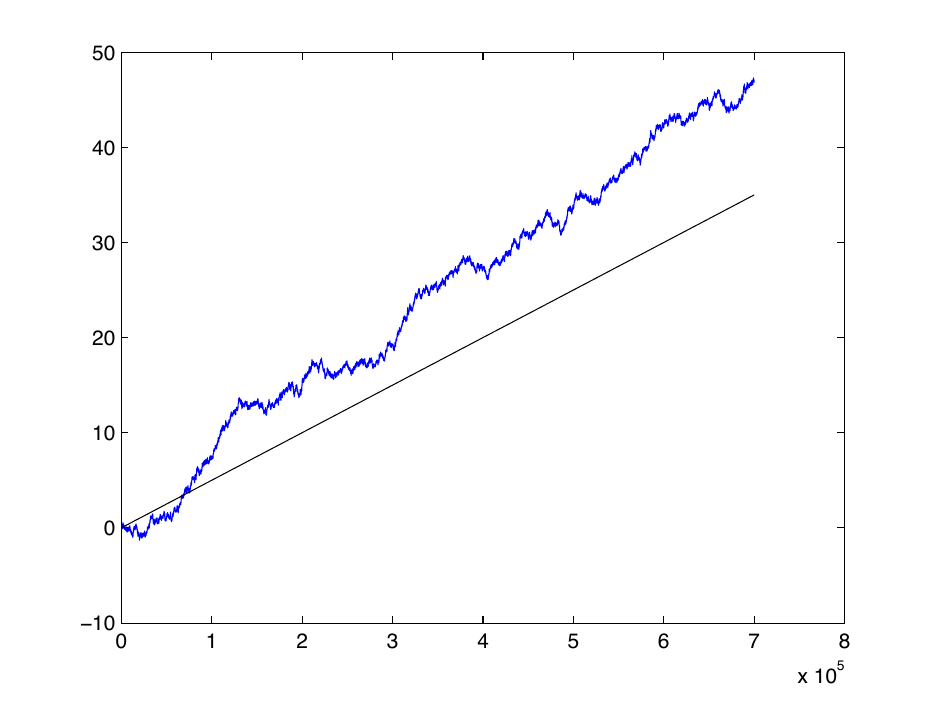}}\quad
			\subfigure[$\sigma=1$]{\includegraphics[width=.3\textwidth]{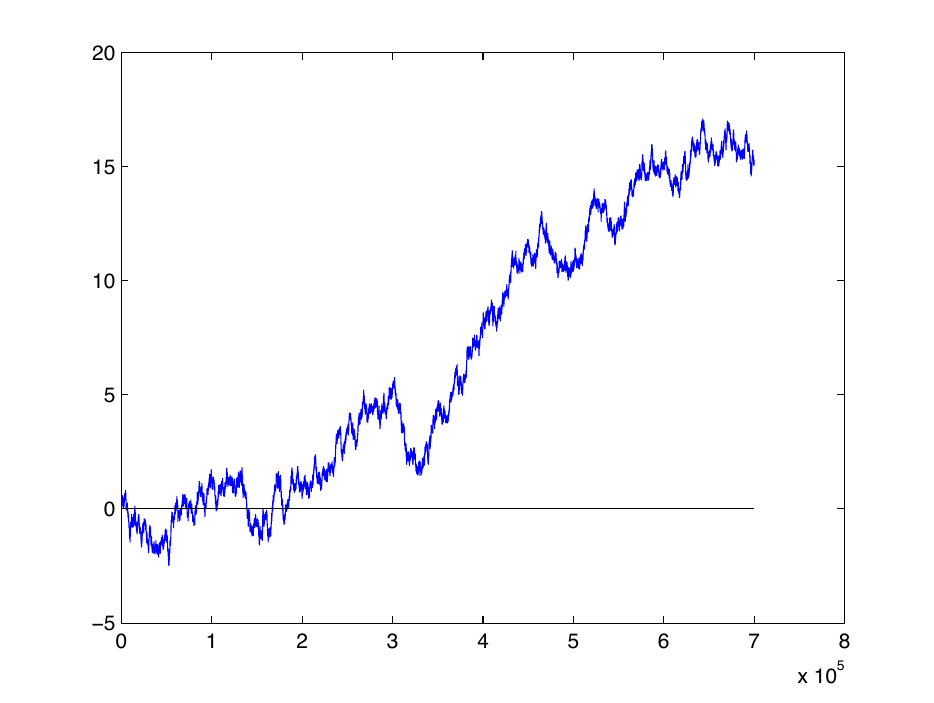}}\quad
			\subfigure[$\sigma=1.5$]{\includegraphics[width=.3\textwidth]{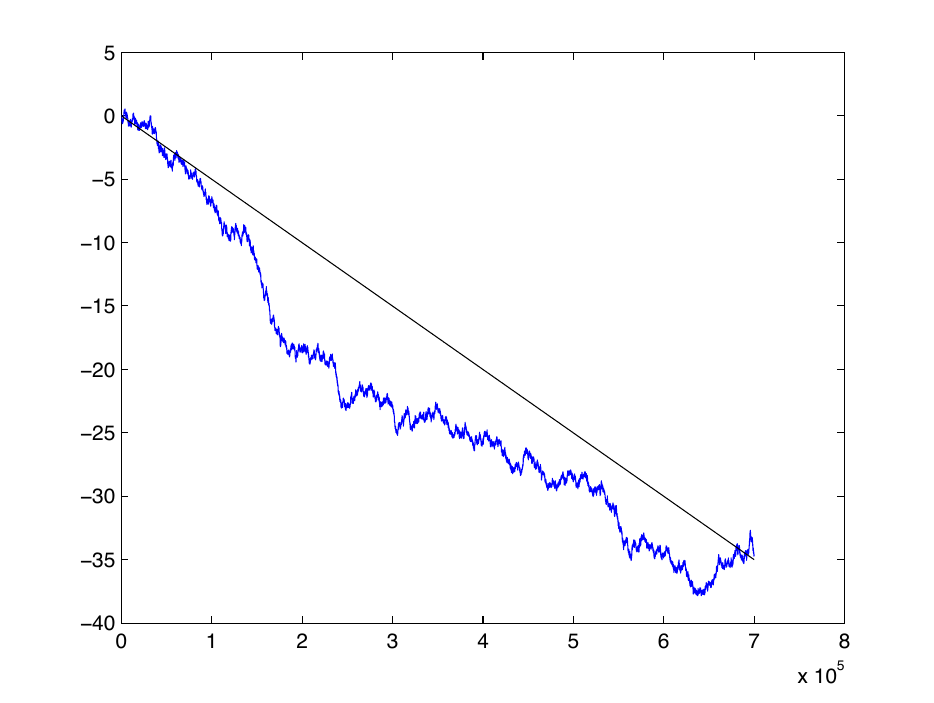}}
		\caption{Noise-induced inversion of the rotation pulsation: $\alpha=1$, $\beta=4$, $\mu=1$ for different values of $\sigma$. Blue: phase ($\theta_t$) as a function of time and black: deterministic pulsation $1-\sigma\mu$. (a) counter-clockwise rotation, (b): no rotation, trajectories appear stochastic and the phase is a Brownian motion, and (c): clockwise rotation.}
		\label{fig:changeOscillation}
	\end{figure}
	We observe in the purely imaginary noise case ($\sigma=0$), $R_t$ satisfies the deterministic pitchfork bifurcation equation. For $\beta>\frac{\mu^2}{2}$, $R_t=\sqrt{\beta +\frac{\mu^2}{2}}$ is the only attractive solution: the solution of the stochastic Hopf bifurcation asymptotically lives on a deterministic circle. Its rotates on this circle with a stochastic phase given by a Brownian motion multiplied by $\mu$ and a deterministic pulsation equal to $1$ (see Figure~\ref{fig:Modulus}). 
	
	\begin{figure}[!h]
		\centering
			\includegraphics[width=.7\textwidth]{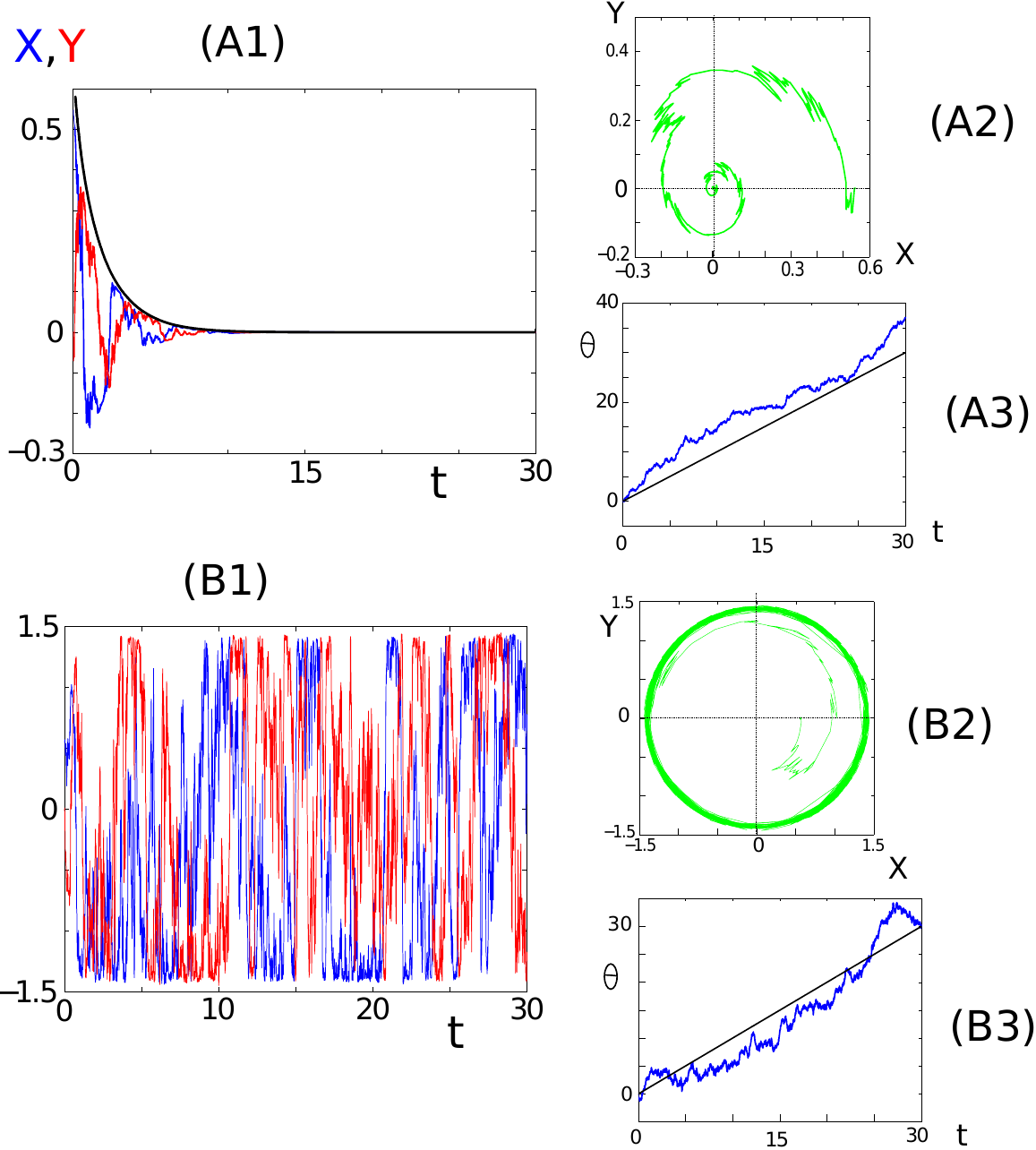}
		\caption{Simulations of sample path trajectories of $(X_t,Y_t)$ for the stochastic Hopf bifurcation with $\alpha=1$, $\sigma=0$ and $\mu=2$ (deterministic modulus case) for different values of $\beta$. (A):$\beta=-2<-\frac{\mu^2}{2}$: zero is a.s. exponentially stable; (A1): trajectories (blue: $X_t$, red: $Y_t$, black: $R_ts$), (A2): phase plane, (A3): argument $\theta_t$ in blue, deterministic drift in black. (B): $\beta=0 > -\frac{\mu^2}{2}$: the modulus converges to a constant value $\sqrt{\mu^2/2}=\sqrt{2}$; (B1): trajectories, (B2): phase plane, (B3): Argument $\theta_t$ (blue), deterministic drift (black).}
		\label{fig:Modulus}
	\end{figure}

Let us eventually note the symmetrical effects of the real and imaginary values of the noise parameter: for fixed $\beta$, $\sigma$ tends to stabilize $0$, whereas $\mu$ has the effect of destabilizing $0$. These two effects compensate each other, and when choosing $\sigma=\mu$, the loss of stability of zero arises at $\beta=0$ precisely as it is the case of the deterministic bifurcation. However, cycles appear for $\beta>\sigma^2/2$ in that case, and hence we observe a systematic delay in the apparition of cycles for $\sigma\neq 0$. 

	We hence conclude that though obviously perturbed by the presence of noise, the global behavior observed in the stochastic Hopf bifurcation with linear noise resembles the deterministic picture: as $\beta$ is increased, a limit cycle appears, around which the solutions stochastically rotate. The picture will be relatively different in the nonlinear noise case $\alpha\neq 1$. In that case we prove the following:
	\begin{proposition}\label{pro:HopfAlphaNot1}
		In the case of the supercritical ($\varepsilon=-1$) stochastic Hopf bifurcation, we have:
		\begin{itemize}
			\item For $\alpha<1$, the fixed point $0$ is:
			\begin{itemize}
				\item stable in probability for $\vert \mu \vert<\vert \sigma\vert$
				\item unstable in probability for $\mu>\sigma$
			\end{itemize}
			Moreover, any solution reaches zero in finite time.
			\item For $\alpha>1$, the fixed point $0$ is:
			\begin{itemize}
				\item stable in probability for $\beta<-\frac{\mu^2}{2}$
				\item unstable in probability otherwise. Two additional stationary solutions exist in that case.
			\end{itemize}
			The solutions almost surely never reach zero. 
		\end{itemize}
	\end{proposition}

	\begin{remark}
		Let us emphasize the fact that the case $\alpha<1$ is substantially different from the general case: indeed, we had seen that the singular point was always stochastically stable. Surprisingly here the stability of zero depends on the relative values of the real and imaginary parts of the noise. 
	\end{remark}
	\begin{proof}
		The case $\alpha>1$ falls in the general analysis developed in section~\ref{sec:general}. Indeed, the diffusion is of the form of equation~\eqref{eq:generalg} with $g(x)=\frac{\mu^2}{2} x^{2\alpha-1} -x^3$ which clearly satisfies assumptions~\ref{assumption:H1} and ~\ref{assumption:H2}. The general analysis hence applies and directly leads to the conclusion of the proposition.

		In the case $\alpha<1$, the equation was not covered in the general analysis, since the drift includes a term $\frac{\mu^2}{2} x^{2\alpha-1}$ and $2\alpha-1<1$. However, the same principles may be applied to that case. First of all, existence and uniqueness of solutions is proved using the same arguments, and the a.s. boundedness of the process as well. Hence there exist a unique solution to the stochastic Hopf bifurcation process with $\alpha<1$, which is almost surely bounded for all times. The stability of zero is investigated as follows. For $\mu>\sigma$, the Lyapunov function $V(x)=-log(x)$ is $C_0^2$ around zero, positive, diverges at zero, and 
		\[\Lop V(x)= \left(-\frac{\mu^2}{2}+\frac{\sigma^2}{2}\right)x^{2\alpha-2} + o(x^{2\alpha-2})\]
		and is hence negative around zero, proving that zero is unstable for $\mu>\sigma$ in probability.

		For $\mu<\sigma$, we define the Lyapunov function $V(x)=x^{\gamma}$ with $\gamma=\frac{\sigma^2-\mu^2}{2\sigma^2}>0$. This function is $C_0^2$ and we have:
		\begin{align*}
			\Lop V(x) &= \gamma \left(\frac{\mu^2}{2}+(\gamma-1)\frac{\sigma^2}{2}\right)x^{2\alpha-2+\gamma} + o(x^{2\alpha-2+\gamma})\\
			& = \gamma\frac{\mu^2-\sigma^2}{4}x^{2\alpha-2+\gamma} + o(x^{2\alpha-2+\gamma}) 
		\end{align*}
		and hence is negative close of zero. Using the same techniques as in the general theory, it is easy to conclude that the singular point is unstable in probability. 
		In order to prove that the solutions almost surely hit zero in finite time, we proceed as in the proof of proposition~\ref{pro:firsthittingtimezero}. We know that the process is almost surely bounded, and analyze Feller's scale functions. It is clear that $p(x)\to\infty$ as $x\to\infty$ following the proof of proposition~\ref{pro:firsthittingtimezero}: the proof only uses the higher order term $-x^3$. We hence need to ensure that $\lim_{x\to 0^+} v(x)<\infty$. This property is related to fine properties of the scale functions close from zero, which is now governed by the term $\mu^2/2 \;x^{2\alpha-1}$. We hence have, around $0$:
		\[G(x)=\frac{\mu^2}{2\sigma^2}\log(x) +\frac{\lambda}{\sigma} \frac{x^{2-2\alpha}}{2(1-\alpha)} -\frac{x^{2(2-\alpha)}}{2(2-\alpha)\sigma^2}\]
		and therefore we can obtain the explicit expression of $p(x)$:
		\[p(x)=\int_c^x y^{-\frac{\mu^2}{\sigma^2}}\exp\left(-\frac{\lambda}{\sigma^2(1-\alpha)}y^{2-2\alpha} + \frac{2y^{4-2\alpha}}{(4-2\alpha)\sigma^2}\right)dy\]
		More importantly, in order to characterize $v$, we consider the term $G(y)-G(z)$ which is equivalent close from $0$ to: $G(y)-G(z)\sim_{0^+} \frac{\mu^2}{2\sigma^2}\log(\frac{y}{z})$. This implies that:
		\begin{align*}
			\int_y^x \exp(2G(y)-G(z))\,dz &\sim_{0^+} \int_y^x \exp\left(\frac{\mu^2}{\sigma^2}\log(\frac{y}{z})\right)\,dz \\
			&\sim \frac{1}{1-\mu^2/\sigma^2} \left(\left(\frac y x\right)^{\mu^2/\sigma^2}-1\right)
		\end{align*}
		and hence:
		\[v(x)\sim \frac{2}{\sigma^2} (\frac{x^2}{\mu^2/\sigma^2 +1}-x)\]
		This function will hence remain bounded as $x\to\infty$, ensuring using Feller's test for explosion that the first exit time of the interval $(0,\infty)$ is almost surely bounded. The almost sure boundedness of the process allows concluding that the probability first hitting time of zero is almost surely finite, which ends the proof. 
	\end{proof}

	Therefore, similarly to the case of the stochastic pitchfork bifurcation, we observed that for $\alpha>1$, the noise does not affect the stability of $0$ and the loss of stability precisely appears for $\beta=0$ as in the deterministic case. The behavior of the system beyond the bifurcation point depends on the shape of the stationary distribution, and in that case again the solution reaches a maximum for a positive value, reproducing qualitatively the behavior of the deterministic pitchfork bifurcation. 

	For $\alpha<1$, the almost sure absorption at zero of the solution persists. However, the stochastic stability of the fixed point zero now depends on the relative values of $\mu$ and $\sigma$. This distinction is not fundamental as long as we are concerned with sample path properties, and similarly to the general case, almost any sample path will eventually be absorbed at zero in finite time.

	Let us eventually remark that the universality properties of this normal form are harder to obtain than in the case of the pitchfork and the saddle-node bifurcation, first because we are in higher dimensions, but also because generic reduction to normal form involves nonlinear changes of parameters, which was neither the case in the fold nor in the pitchfork bifurcation (see~\cite{kuznetsov:98}). In our stochastic context, such nonlinear changes of coordinates do not have the same effect as in the deterministic case because of the stochastic terms affecting the drift through the It\^o formula. Let us however give an idea of the kind of calculations possible to identify stochastic Hopf bifurcations. To this end, let us consider a two-dimensional system:
	\[dX_t=F(X_t)\,dt+G(X_t)dW_t\]
	where $G(x)$ is precisely of the complex noise form considered above. The modulus of the solution $R_t$ satisfies the SDE:
	\[dR_t=\frac 1 {R_t} \left(X^1_t F^1(X_t)+X^2_t F^2(X_t) + \frac{\mu^2}{2}R_t\right)\,dt + (X^1_t G^1(X_t)+X^2_t G^2(X_t))dW_t\]
	where $X^1,F^1$ and $G^1$ represent the first component of the vector field. The question that now arises is whether this system can be reduced to a pitchfork equation with possibly higher order terms, which can be checked in a case-by-case basis. 

	\begin{example}
		Let us consider a Wilson and Cowan neural network composed of an excitatory and an inhibitory population. The network equations in that case read:
	\begin{equation}\label{eq:WCHopf}
		\begin{cases}
				d\nu^1_t = (-\nu^1_t + S(g \nu^1_t) - S(g \nu^2_t))+(\sigma \nu^1_t-\mu \nu^2_t)\vert \nu_t \vert^{\alpha-1}dW_t\\
				d\nu^2_t = (-\nu^2_t + S(g \nu^1_t) + S(g \nu^2_t))+(\sigma \nu^2_t+\mu \nu^1_t)\vert \nu_t \vert^{\alpha-1}dW_t
			\end{cases}
	\end{equation}
		with $S(x)=\tanh(g\,x)$ and $\vert \nu_t \vert$ is the modulus of $(\nu^1,\nu^2)$. In the deterministic case, it is straightforward to show that the system undergoes a Hopf bifurcation at $g=1$. Let us now consider the equation satisfied by the modulus $\rho_t=\vert \nu_t\vert$. Using It\^o's formula we obtain:
		\[d\rho_t= \Big((-1+g)\rho_t+\mu^2/2\rho_t^{2\alpha-1} + g(\nu^1_t,\nu^2_t)\Big)\,dt+\sigma \rho_t^{\alpha}dW_t\]
		 with 
		\begin{align*}
			g(x,y) &= -\frac{g^3}{3\rho} (x^4+y^4+x^3y+xy^3)+O(\rho^4) \\
			& = -\frac{g^3}{3\rho} ((x+y)^4 -2x^2y^2)+O(\rho^4) \leq  -\frac{g^3}{12}\rho^3 +O(\rho^4)
		\end{align*}
		The argument satisfies the equation:
		\[d\theta_t= \big(g  - \sigma\mu\rho^{2\alpha-2} + O(\rho_t^2)\big)\,dt+ \mu\rho^{\alpha-1}dW_t\]
		The dynamics of the system can hence be analyzed using the above analysis:
		\begin{itemize}
			\item $\alpha=1$, zero is stable in probability when $(-1+g+\mu^2/2)\leq \sigma^2/2$ and unstable otherwise. At this point, a stochastic Hopf bifurcation arises, and oscillations appear, with pulsation locally equivalent to $g-\mu\sigma$ (with order two corrections in $\rho$) and stochastic phase given by a Brownian motion multiplied by $\mu$. The oscillation orientation bifurcation takes place for $g=\mu\sigma$. 
			\item For $\alpha>1$ zero is stable if and only if $(-1+g+\mu^2/2)\leq 0$.
			\item For $\alpha<1$, the fixed point is stable in probability if $\vert \mu\vert <\vert \sigma \vert$ and unstable otherwise. The first hitting time of the singular point is almost surely finite in both cases. 
		\end{itemize} 
		These results are illustrated in Fig.~\ref{fig:WCHopf}
		\begin{figure}[htbp]
			\begin{center}
				\subfigure[$\alpha=0.8$, $\mu=2$, $\sigma=1$]{\includegraphics[width=.3\textwidth]{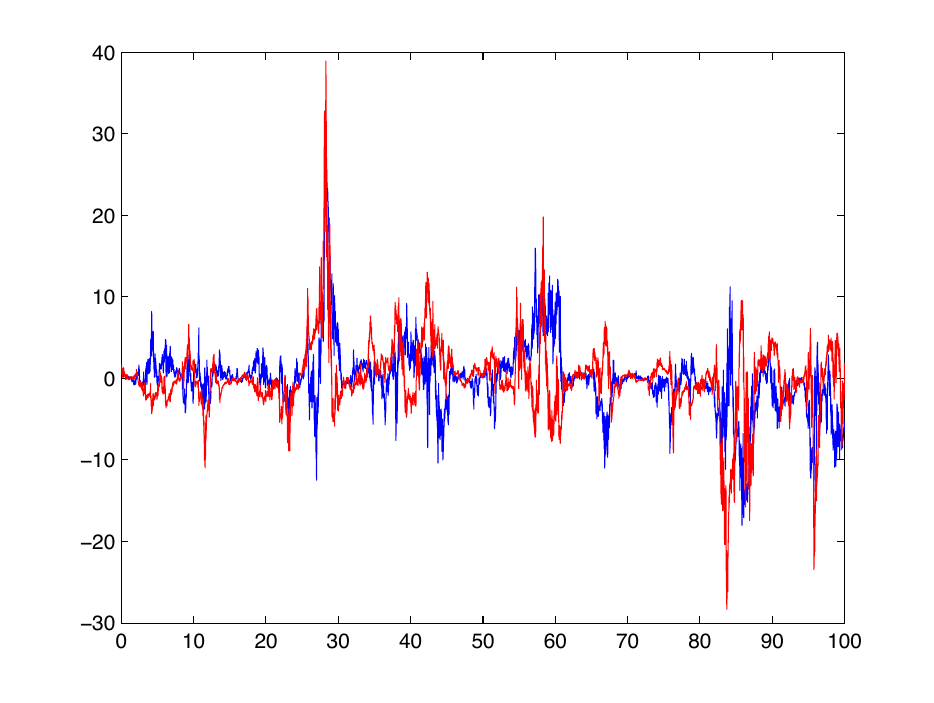}}\quad
				\subfigure[$\alpha=1$, $\mu=2$ $\sigma=1$]{\includegraphics[width=.3\textwidth]{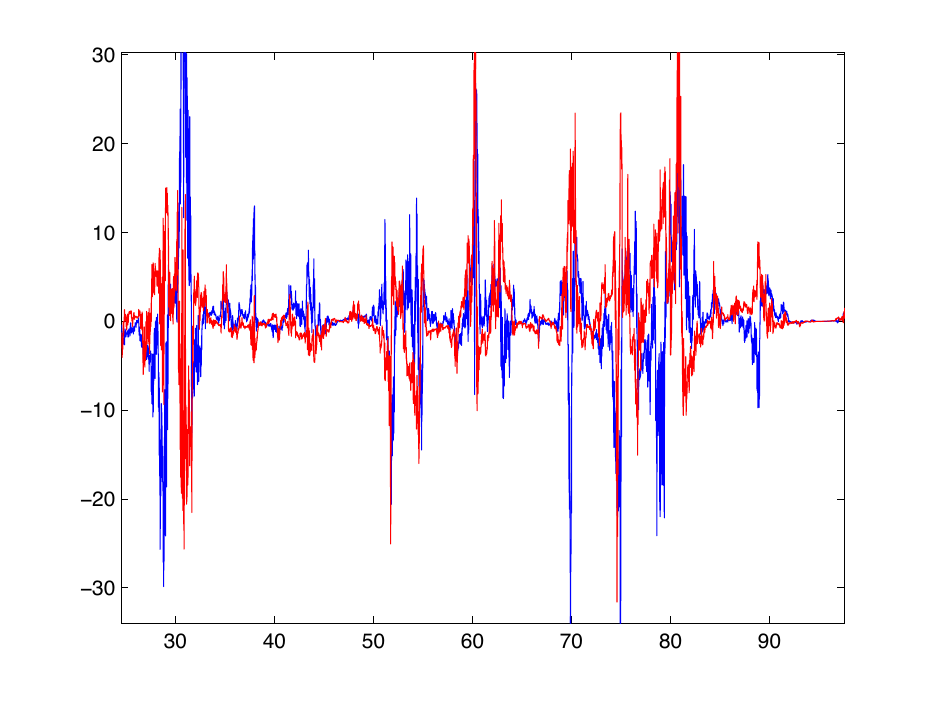}}\quad
				\subfigure[$\alpha=1$, $\mu=1$ $\sigma=0.5$]{\includegraphics[width=.3\textwidth]{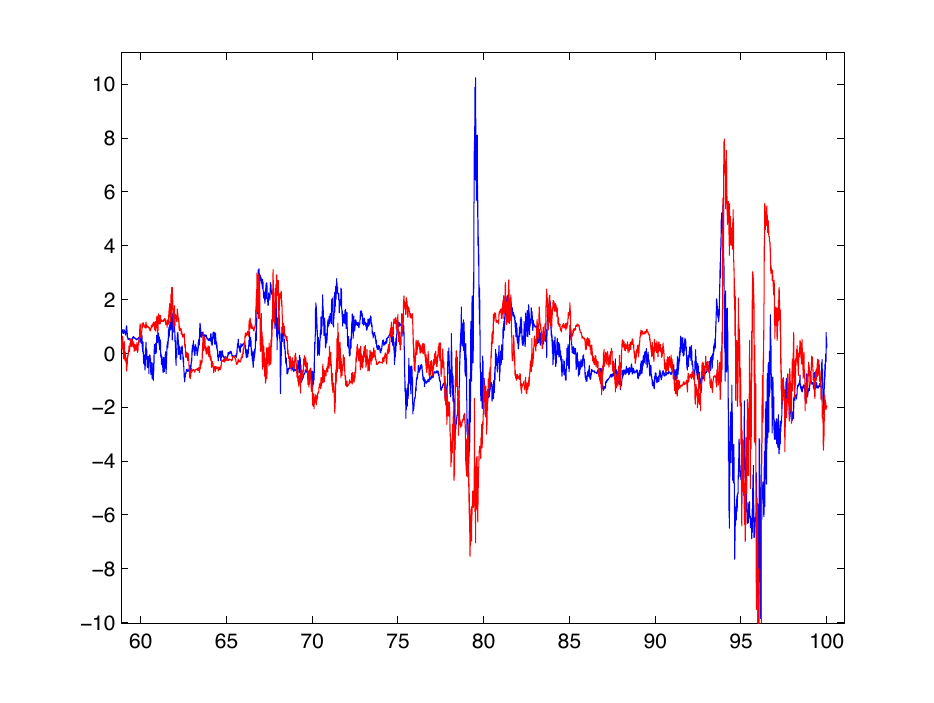}}\\
				\subfigure[$\alpha=0.8$, $\mu=1$, $\sigma=2$]{\includegraphics[width=.3\textwidth]{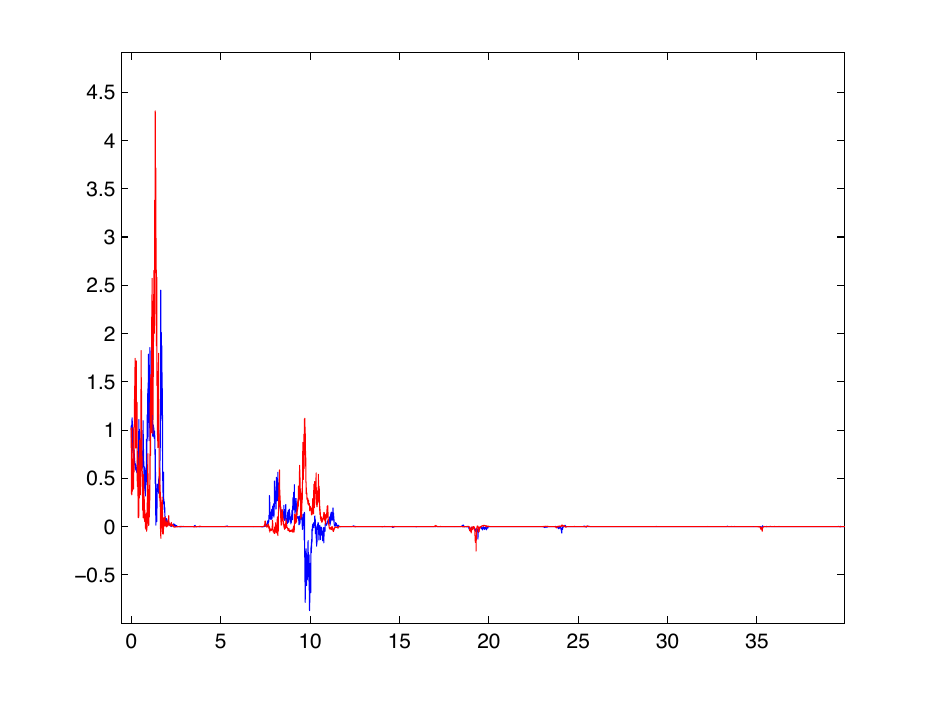}}\quad
				\subfigure[$\alpha=1$, $\mu=2$ $\sigma=1$]{\includegraphics[width=.3\textwidth]{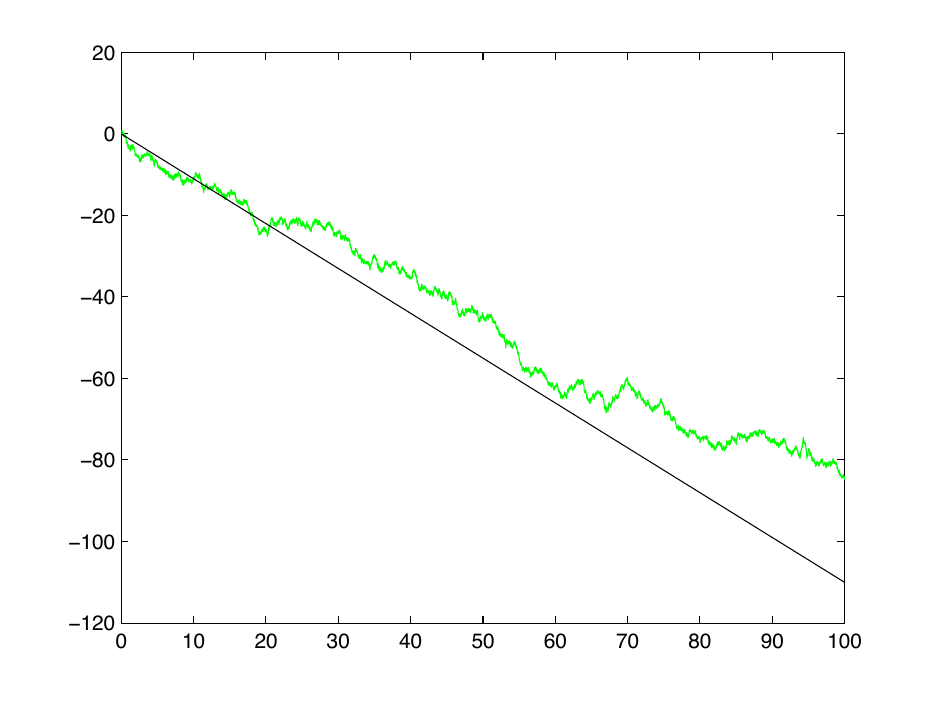}}\quad
				\subfigure[$\alpha=1$, $\mu=1$ $\sigma=0.5$]{\includegraphics[width=.3\textwidth]{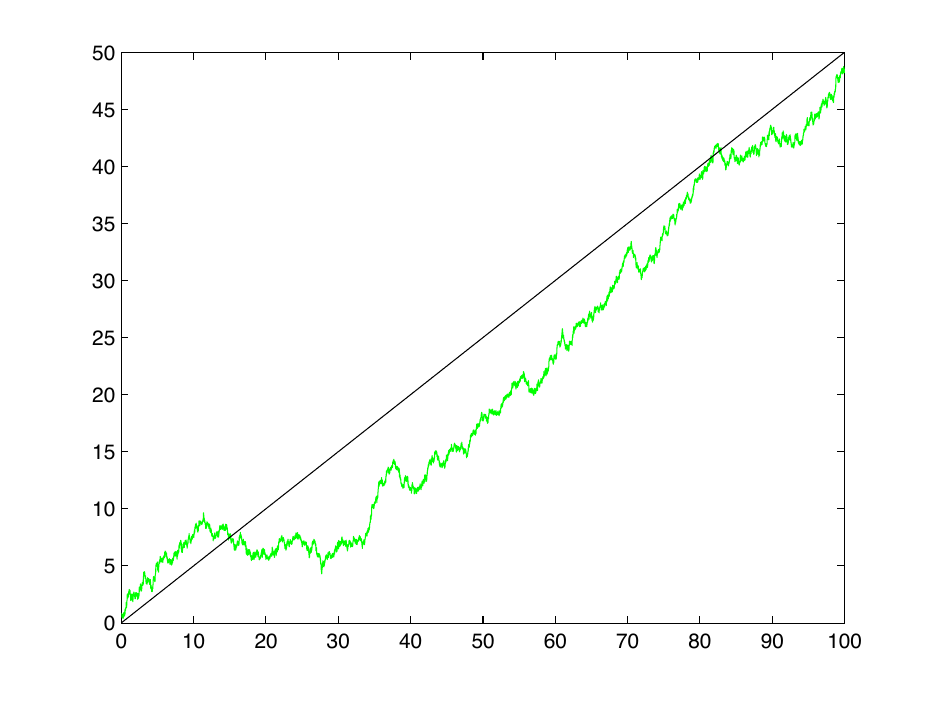}}
			\end{center}
			\caption{Dynamics of the Wilson and Cowan two-populations system~\eqref{eq:WCHopf}. (a) and (d) correspond to $\alpha<1$: for $\vert \mu\vert <\vert\sigma\vert$, the singular point is stable (d) and unstable otherwise (a). The other figures correspond to $\alpha=1$ and illustrate the change of rotation orientation bifurcation at $\sigma\mu=g$. All the figures correspond to $g=1$. }
			\label{fig:WCHopf}
		\end{figure}
	\end{example}


\end{document}